\documentclass{article}
\usepackage{amsmath,amssymb,amsthm}
\usepackage[cp1250]{inputenc}
\usepackage{enumerate}
\usepackage{a4wide}

\numberwithin{equation}{section}

\usepackage{graphicx}
\usepackage{subfig}

\newcommand{\abs}[1]{\ensuremath{\left| #1 \right|}}
\newcommand{\ep}{\varepsilon}
\newcommand{\id}{\mathbb{I}}
\newcommand{\Sz}{\mathcal{S}^{2\times 2}_0}
\newcommand{\R}{\mathbb{R}} 
\newcommand{\Div}{{\rm div}_\vx} 
\newcommand{\vr}{\varrho}
\newcommand{\half}{\frac{1}{2}} 
\newcommand{\vc}[1]{{\bf #1}}
\newcommand{\vv}{\vc{v}}
\newcommand{\vx}{\vc{x}} 
\newcommand{\Grad}{\nabla_\vx}
\newcommand{\dx}{\,{\rm d}\vx}
\newcommand{\dt}{\,{\rm d}t}
\newcommand{\intx}[1]{\int_{\R^2}#1\dx}
\newcommand{\intt}[1]{\int_0^\infty #1\dt}
\newcommand{\DC}{C^\infty_{\rm c}} 
\newcommand{\bfphi}{\boldsymbol{\varphi}}
\newcommand{\ov}[1]{\overline{#1}} 
\newcommand{\un}[1]{\underline{#1}} 
\newcommand{\til}[1]{\widetilde{#1}} 
\newcommand{\new}{{\rm new}}

\newtheorem{theorem}{Theorem}[section]
\newtheorem{lemma}[theorem]{Lemma}
\newtheorem{definition}[theorem]{Definition}
\newtheorem{propo}[theorem]{Proposition}
\newtheorem{remark}[theorem]{Remark}

\title{Non-uniqueness of admissible weak solution to the Riemann problem for the full Euler system in 2D}
\author{Hind Al Baba$^1$\footnote{albaba@math.cas.cz} \and Christian Klingenberg$^2$\footnote{klingen@mathematik.uni-wuerzburg.de} \and Ond\v rej Kreml$^1$\footnote{kreml@math.cas.cz} \and
V\'aclav M\'acha$^1$\footnote{macha@math.cas.cz} \and Simon Markfelder$^2$\footnote{simon.markfelder@mathematik.uni-wuerzburg.de}}

\begin{document}
\maketitle

\bigskip

\centerline{$^1$Institute of Mathematics of the Czech Academy of Sciences}

\centerline{\v Zitn\' a 25, CZ-115 67 Praha 1, Czech Republic}
\medskip

\centerline{$^2$Department of Mathematics, W\"urzburg University}

\centerline{Emil-Fischer-Str. 40, 97074 W\"urzburg, Germany}

\bigskip

\textbf{Abstract:} The question of well- and ill-posedness of entropy admissible solutions to the multi-dimensional systems of conservation laws has been studied recently in the case of isentropic Euler equations. In this context special initial data were considered, namely the 1D Riemann problem which is extended trivially to a second space dimension. It was shown that there exist infinitely many bounded entropy admissible weak solutions to such a 2D Riemann problem for isentropic Euler equations, if the initial data give rise to a 1D self-similar solution containing a shock. In this work we study such a 2D Riemann problem for the full Euler system in two space dimensions and prove the existence of infinitely many bounded entropy admissible weak solutions in the case that the Riemann initial data give rise to the 1D self-similar solution consisting of two shocks and possibly a contact discontinuity.

\tableofcontents 

\section{Introduction} 
In this paper we consider the full compressible Euler system in the whole two-dimensional space, i.e. 
\begin{equation}
\begin{split}
\partial_t \vr + \Div (\vr\vv) &= 0, \\
\partial_t (\vr\vv) + \Div(\vr\vv\otimes \vv)+\Grad p  &= 0, \\
\partial_t \bigg(\half\vr|\vv|^2 + \vr\,e(\vr,p)\bigg) + \Div \bigg[\bigg(\half\vr|\vv|^2 + \vr\,e(\vr,p) + p\bigg)\vv\bigg] &= 0,
\end{split}
\label{eq:euler}
\end{equation}
where the density $\vr=\vr(t,\vx)\in\R^+$, the pressure $p=p(t,\vx)\in\R^+$ and the velocity $\vv=\vv(t,\vx)\in\R^2$ are unknown functions of the time $t\in[0,\infty)$ and the position $\vx=(x_1,x_2)\in\R^2$. 

Furthermore we consider an ideal gas, i.e. 
\[
e(\vr,p) = c_v\,\frac{p}{\vr},
\]
where $c_v>0$ is a constant called the specific heat at constant volume. We recall that from a physical point of view $c_v=\frac{f}{2}$, where $f$ is the number of degrees of freedom.

The system of equations \eqref{eq:euler} is complemented with the initial condition
\begin{equation}\label{eq:initial}
(\vr,\vv,p)(0,\vx) = (\vr^0,\vv^0,p^0)(\vx) \qquad \text{ in } \R^2.
\end{equation}

Further, we supplement the equations \eqref{eq:euler} with the entropy condition 
\begin{equation} \label{eq:entropy}
\partial_t\Big(\vr\,s(\vr,p)\Big) + \Div\Big(\vr\,s(\vr,p)\vv\Big) \geq 0,
\end{equation}
where $s(\vr,p)$ is the entropy. For the ideal gas we have 
\[
s(\vr,p)=\log\Big(\frac{p^{c_v}}{\vr^{c_v+1}}\Big).
\]
By \emph{admissible weak solution} - we will also use the term \emph{weak entropy solution} - we mean a triple of $L^\infty$-functions $(\vr,\vv,p)$ that fulfill \eqref{eq:euler}-\eqref{eq:entropy} in the sense of distributions, see Section \ref{s:main} for precise definitions. 

The system \eqref{eq:euler} can be formulated also with unknown temperature $\vartheta$ instead of pressure $p$. We note that the temperature can be reconstructed from the pressure and the density by $\vartheta = \frac{p}{\varrho}$ as long as the density $\vr$ is positive which is always the case in this paper.

We consider Riemann initial data
\begin{equation}
\begin{split}
(\vr^0,\vv^0,p^0)(\vx):=\left\{
\begin{array}[c]{ll}
(\vr_-,\vv_-,p_-) & \text{ if }x_2<0 \\
(\vr_+,\vv_+,p_+) & \text{ if }x_2>0
\end{array}
\right. ,
\end{split}
\label{eq:Riemann}
\end{equation}
where $\vr_\pm\in\R^+$, $\vv_\pm\in\R^2$ and $p_\pm\in\R^+$ are constants. We write $\vv=(v_{1},v_{2})^T$ for the components of the velocity $\vv$. Throughout this paper we additionally assume $v_{-,1}=v_{+,1}=0$, where $v_{\pm, 1}$ are the first components of the velocities $\vv_\pm$.

Following the groundbreaking results of De Lellis and Sz\'ekelyhidi \cite{DLSz1,DLSz2} concerning ill-posedness of the incompressible Euler equations, the question of well- and ill-posedness of admissible weak solutions to the isentropic compressible Euler equations was raised. In fact, De Lellis and Sz\'ekelyhidi themselves showed in \cite{DLSz2} the existence of initial data $(\vr^0,\vv^0) \in L^\infty(\R^2)$ for which there exist infinitely many admissible weak solutions (such initial data is called \textit{wild initial data}). In the context of hyperbolic conservation laws this result in particular shows that the notion of  admissible weak solution, i.e. solution satisfying the system of conservation laws and an inequality arising from the companion law, is probably not suitable for the problem in the case of more than one space dimension.

This result was further improved by Chiodaroli \cite{Ch} and Feireisl \cite{Fe} who showed that the ill-posedness is related to the irregularity of the velocity field, proving that there exist wild initial data $(\vr^0, \vv^0)$ with $\vr^0 \in C^1(\R^2)$. Further, De Lellis, Chiodaroli and Kreml \cite{ChiDelKre15} showed the existence of Lipschitz wild initial data and finally, Chiodaroli et al. \cite{CKMS} proved the existence of $C^\infty$ wild initial data. The method of the proof was motivated by the observation of Sz\'ekelyhidi \cite{sz} that for the incompressible Euler system, vortex sheet initial data are wild. Hence, the authors in \cite{ChiDelKre15} studied the Riemann problem for the isentropic Euler system and found an example of Riemann initial data allowing for existence of infinitely many admissible weak solutions while being generated by a backwards rarefaction wave (i.e. a compression wave). The authors in \cite{CKMS} used a similar idea and adjusted the theory to handle data generated by a smooth compression wave.

Since the Riemann problem is a basic building block in the theory of one-dimensional systems of conservation laws, it was only natural to study its properties in multi-dimensional case further. In fact, Chen and Chen \cite{Chen} already proved that if the solution to the Riemann problem consists only of rarefaction waves, such solution is unique in the class of multi-dimensional admissible weak solutions. The same was later proved also in \cite{FeiKre}.

On the non-uniqueness side, Chiodaroli and Kreml \cite{ChiKre14} proved that whenever the Riemann initial data give rise to the self-similar solution consisting of two shocks, there exist also infinitely many other bounded admissible weak solutions. Moreover they also showed that the self-similar solution may not be entropy rate admissible, meaning that some of the other solutions constructed may produce more total entropy. The study of the Riemann problem for the isentropic Euler system continued with results of Chiodaroli and Kreml \cite{ChiKre17} and Klingenberg and Markfelder \cite{MarKli17} who independently showed ill-posedness in the case of Riemann initial data giving rise to the self-similar solution consisting of one rarefaction wave and one shock. More precisely in \cite{ChiKre17} the authors need a certain smallness condition for the initial data, whereas the result in \cite{MarKli17} covers all cases of such Riemann initial data and also solves the case of Riemann initial data giving rise to self-similar solution consisting of a single shock. 

Finally, B\v rezina, Chiodaroli and Kreml \cite{BrChKr} showed that the above mentioned results of ill-posedness of admissible weak solutions hold also in the presence of a contact discontinuity in the self-similar solution arising from the multidimensionality of the problem. However, the question, whether a single contact discontinuity is a unique admissible weak solution to the isentropic Euler system or not, is still an open problem; it is worth mentioning that this is actually a direct analog of the vortex sheet initial data considered by Sz\'ekelyhidi in \cite{sz}.

In view of results for the isentropic system a natural question arises, whether these results can be transferred to the case of full Euler equations. We note the result of Markfelder and Klingenberg \cite{MaKl2}, who showed that for the isentropic system there exist infinitely many solutions conserving the energy. This however can be viewed as a proof of existence of infinitely many solutions to the full system with a proper choice of $c_v$; all these solutions also conserve the entropy. In addition to that we want to mention the result by Feireisl et al. \cite{FeKlKrMa}, who proved existence of wild initial data for the full Euler equations. Their solutions conserve the entropy and have piecewise constant density and pressure. 

As in the isentropic case, we aim for the classification of 1D self-similar solutions to the Riemann problem for the full Euler system from the point of view of their (non)uniqueness. Table \ref{table} shows all the possible cases for the 1D self-similar solution. It was shown by Chen and Chen in \cite{Chen} that uniqueness of rarefaction waves actually holds also for the full system, see also \cite{FeiKreVas}, these works cover the cases 1, 3, 7 and 9 in Table \ref{table}. In this paper we present the first step in the direction of non-uniqueness. More precisely, we show that if the 1D self-similar solution to the Riemann problem consists of two shocks and possibly a contact discontinuity, then there exist also infinitely many bounded admissible weak solutions. This covers the cases 5 and 14 in Table \ref{table} and therefore serves as the extension of the result of Chiodaroli and Kreml \cite{ChiKre14} to the full Euler system. Since the theory for the cases involving one shock and one rarefaction wave is technically more difficult, it will be treated in following works.

The paper is structured as follows. In Section \ref{s:1D} we recall some results for the one-dimensional Riemann problem for the full Euler system and present a condition on the initial states $(\vr_\pm, \vv_\pm,p_\pm)$ such that the self-similar solution consists of two shocks and possibly a contact discontinuity. In Section \ref{s:main} we present our two main results, Theorem \ref{thm:mainWUE} and Theorem \ref{t:main}. In Section \ref{s:exis} we provide the technical tools used in the proof, in particular we present a key lemma relating the existence of infinitely many admissible weak solutions to the existence of a single so-called admissible fan subsolution. Finally, in Sections \ref{s:proofA} and \ref{s:proofB} we present proofs of our two main results.

\section{The 1D Riemann solution}\label{s:1D}

The initial data \eqref{eq:initial} of the problem we study are one-dimensional, we have $(\vr^0,\vv^0,p^0)(x_2)$ only with no dependence on the $x_1$ variable. Therefore it is reasonable to search for solutions which share this property, i.e. for solutions satisfying $(\vr,\vv,p)(t,\vx) = (\vr,\vv,p)(t,x_2)$. Moreover since we assume $v_{-,1} = v_{+,1} = 0$, we search for solutions with the property $v_1(t,\vx) \equiv 0$.

Hence we obtain the 1D Riemann problem
\begin{equation} \label{eq:1dEuler}
\begin{split}
\partial_t \vr + \partial_2 (\vr v_2) &= 0, \\
\partial_t (\vr v_2) + \partial_2 \big(\vr (v_2)^2 + p\big) &= 0, \\
\partial_t \bigg(\half\vr(v_2)^2 + \vr\, e(\vr,p)\bigg) + \partial_2 \bigg[\bigg(\half\vr(v_2)^2 + \vr\, e(\vr,p) + p\bigg)v_2\bigg] &= 0,
\end{split}
\end{equation}
for the unknown functions $(\vr,v_2,p)(t,x_2)$ with initial data
\begin{equation} \label{eq:1dInitial}
\begin{split}
(\vr,v_2,p)(0,x_2) &= (\vr^0,v_2^0,p^0)(x_2):=\left\{
\begin{array}[c]{ll}
(\vr_-,v_{-,2},p_-) & \text{ if }x_2<0 \\
(\vr_+,v_{+,2},p_+) & \text{ if }x_2>0
\end{array}
\right. .
\end{split}
\end{equation}

The problem \eqref{eq:1dEuler}-\eqref{eq:1dInitial} is a model example of a one-dimensional Riemann problem for a system of conservation laws and has been studied extensively by several authors, e.g. by Smoller \cite[Chapter 18 \S B]{Smoller67}. It is well known that this problem admits a self-similar solution (i.e. solution consisting of functions of a single variable $\frac{x_2}{t}$) belonging to the class of BV functions. This solution to \eqref{eq:1dEuler}-\eqref{eq:1dInitial} is unique in the class of BV self-similar functions and consists of three waves connected by constant states, where the first wave is either an admissible shock or a rarefaction wave, the second wave is a contact discontinuity and the third wave is again either an admissible shock or a rarefaction wave. In the rest of this paper we will call this solution the \textit{1D Riemann solution}. Table \ref{table} shows all the possibilities of the structure of the 1D Riemann solution.
\renewcommand{\arraystretch}{1.3}
\begin{table}[h]
	\centering 
	\begin{tabular}{|c|c|c|c|c|c|c|c|c|} \cline{2-4} \cline{7-9}
		\multicolumn{1}{c|}{}& \centering 1-wave & \centering 2-wave & \centering 3-wave & \multicolumn{2}{c|}{} & \centering 1-wave & \centering 2-wave & \centering 3-wave \tabularnewline \cline{2-4} \cline{7-9} \multicolumn{7}{c}{}\\[-4.7mm] \cline{1-4} \cline{6-9}
		1 & \centering - & \centering - & \centering - & & 10 & \centering - & \centering contact & \centering - \tabularnewline \cline{1-4} \cline{6-9} 
		2 & \centering - & \centering - & \centering shock & \hspace{6mm} $ $& 11 & \centering - & \centering contact & \centering shock \tabularnewline \cline{1-4} \cline{6-9}
		3 & \centering - & \centering - & \centering rarefaction & & 12 & \centering - & \centering contact & \centering rarefaction \tabularnewline \cline{1-4} \cline{6-9}
		4 & \centering shock & \centering - & \centering - & & 13 & \centering shock & \centering contact & \centering - \tabularnewline \cline{1-4} \cline{6-9} 
		5 & \centering shock & \centering - & \centering shock & & 14 & \centering shock & \centering contact & \centering shock \tabularnewline \cline{1-4} \cline{6-9}
		6 & \centering shock & \centering - & \centering rarefaction & & 15 & \centering shock & \centering contact & \centering rarefaction \tabularnewline \cline{1-4} \cline{6-9}
		7 & \centering rarefaction & \centering - & \centering - & & 16 & \centering rarefaction & \centering contact & \centering - \tabularnewline \cline{1-4} \cline{6-9}
		8 & \centering rarefaction & \centering - & \centering shock & & 17 & \centering rarefaction & \centering contact & \centering shock \tabularnewline \cline{1-4} \cline{6-9}
		9 &\centering rarefaction & \centering - & \centering rarefaction & & 18 & \centering rarefaction & \centering contact & \centering rarefaction \tabularnewline \cline{1-4} \cline{6-9}
	\end{tabular}
	\caption{All the 18 possibilities of the structure of the 1D Riemann solution} \label{table}
\end{table}
\renewcommand{\arraystretch}{1}

It is easy to observe that the 1D Riemann solution is an admissible weak solution to the original 2D problem \eqref{eq:euler}-\eqref{eq:initial}.

In this paper we will focus on the cases, where the 1D Riemann solution contains two admissible shocks and possibly a contact discontinuity. Therefore we recall the following proposition stating conditions for the initial data under which such 1D Riemann solution emerges.

\begin{propo} \label{prop:standard}
	Assume that either
	\begin{itemize}
		\item $p_-\leq p_+$ and $v_{+,2} - v_{-,2} < -(p_+ - p_-)\,\sqrt{\frac{2\,c_v}{\vr_-(p_- + (2c_v + 1) p_+)}}$ or
		\item $p_+\leq p_-$ and $v_{+,2} - v_{-,2} < -(p_- - p_+)\,\sqrt{\frac{2\,c_v}{\vr_+(p_+ + (2c_v + 1) p_-)}}$.
	\end{itemize}
	Then the 1D Riemann solution to problem \eqref{eq:1dEuler}-\eqref{eq:1dInitial} consists of a 1-shock, possibly a 2-contact discontinuity and a 3-shock. The intermediate states $(\vr_{M-},\vv_M,p_M)$ and $(\vr_{M+},\vv_M,p_M)$ are given by
	\begin{itemize}
		\item $p_M$ is the unique solution of 
		\begin{equation*}
		-\sqrt{2\,c_v}\left(\frac{p_M - p_-}{\sqrt{\vr_- (p_- + (2c_v + 1) p_M)}} + \frac{p_M - p_+}{\sqrt{\vr_+ (p_+ + (2c_v + 1) p_M)}}\right) = v_{+,2} - v_{-,2};
		\end{equation*}
		\item $v_{M,2} = v_{-,2} - \sqrt{2\,c_v}\frac{p_M - p_-}{\sqrt{\vr_- (p_- + (2c_v + 1) p_M)}} ;$ 
		\item $\vr_{M-} = \vr_- \frac{p_- + (2c_v + 1) p_M}{p_M + (2c_v + 1) p_-};$
		\item $\vr_{M+} = \vr_+ \frac{p_+ + (2c_v + 1) p_M}{p_M + (2c_v + 1) p_+};$
	\end{itemize}
	and the following properties hold.
	\begin{itemize}
		\item If $\vr_{M-}=\vr_{M+}$ then the 2-contact discontinuity does not appear.
		\item $p_M>\max\{p_-,p_+\}$, $\vr_{M-}>\vr_-$ and $\vr_{M+}>\vr_+$.
		\item The shock speeds denoted by $\sigma_-,\sigma_+$ satisfy $\sigma_\pm(\vr_\pm - \vr_{M\pm}) = \vr_\pm v_{\pm,2} - \vr_{M\pm}v_{M,2}$.
		\item The speed of the contact discontinuity is given by $v_{M,2}$.
	\end{itemize}
\end{propo}

For the proof of Proposition \ref{prop:standard} we refer to \cite{Smoller67}. More details about self-similar solutions to 1D systems of conservation laws can be also found in \cite{Dafermos16}.

\section{Main results}\label{s:main}

In this section we state the definition of solutions to the problem \eqref{eq:euler}-\eqref{eq:initial} and our main results.

\begin{definition} 
	Let $(\vr^0,\vv^0,p^0) \in L^\infty(\mathbb{R}^2)$. The triplet $(\rho,\vv,p) \in L^\infty((0,\infty) \times \mathbb{R}^2;\R^+\times\R^2\times\R^+)$ is called a weak solution of \eqref{eq:euler}-\eqref{eq:initial} if the following system of identities is satisfied
	\begin{equation}\label{eq:weak1}
	\int_0^{\infty} \int_{\R^2} \Big(\vr\partial_t \psi + \vr \vv \cdot \Grad\psi\Big)\dx\dt + \int_{\R^2} \vr^0(\vx)\psi(0,\vx) \dx = 0
	\end{equation}
	for all test functions $\psi \in C^\infty_c([0,\infty) \times \mathbb{R}^2)$,
	\begin{equation}\label{eq:weak2}
	\int_0^{\infty} \int_{\R^2} \Big(\vr\vv\cdot\partial_t \bfphi + (\vr \vv\otimes\vv): \Grad\bfphi + p\, \Div \bfphi\Big)\dx\dt + \int_{\R^2} \vr^0(\vx)\vv^0(\vx)\cdot\bfphi(0,\vx) \dx = 0
	\end{equation}
	for all test functions $\bfphi \in C^\infty_c([0,\infty) \times \mathbb{R}^2;\R^2)$,
	\begin{equation} \label{eq:weak3}
	\begin{split}
	\int_0^{\infty} \int_{\R^2} \left[\left(\frac 12 \vr\abs{\vv}^2 + \vr\, e(\vr,p)\right)\cdot\partial_t \phi + \left(\frac 12 \vr\abs{\vv}^2 + \vr \,e(\vr,p) + p\right) \vv \cdot \Grad\phi \right]\dx\dt & \quad\\ 
	+ \int_{\R^2} \left(\frac 12 \vr^0(\vx)\abs{\vv^0(\vx)}^2 + \vr^0(\vx) e(\vr^0(\vx),p^0(\vx))\right)\phi(0,\vx) \dx &= 0 
	\end{split}
	\end{equation}
	for all test functions $\phi \in C^\infty_c([0,\infty) \times \mathbb{R}^2)$.
\end{definition}

\begin{definition} 
	Let $(\vr^0,\vv^0,p^0) \in L^\infty(\mathbb{R}^2)$. A weak solution of \eqref{eq:euler}-\eqref{eq:initial} is called admissible (or weak entropy solution), if the following inequality is satisfied for all non-negative test functions $\varphi \in C^\infty_c([0,\infty) \times \mathbb{R}^2;\mathbb{R}_0^+)$:
	\begin{equation}\label{eq:weak4}
	\int_0^{\infty} \int_{\R^2} \Big(\vr \,s(\vr,p) \partial_t \varphi + \vr \,s(\vr,p) \vv \cdot \Grad\varphi\Big)\dx\dt + \int_{\R^2} \vr^0(\vx)s(\vr^0(\vx),p^0(\vx))\varphi(0,\vx) \dx \leq 0.
	\end{equation}
\end{definition}

The solutions which we are going to construct will have a special structure, partially motivated by the structure of the 1D Riemann solution. Therefore we introduce first the notion of $N$-fan partition ($N = 1,2$) and then $N$-fan solutions.

\begin{definition} 
	Let $N\in\{1,2\}$ and $\mu_0<\mu_1$ ($\mu_0<\mu_1<\mu_2$) real numbers. An $N$-fan partition of $(0,\infty)\times\R^2$ is a set of $N+2$ open sets $\Omega_-,\Omega_1,\Omega_+$ ($\Omega_-,\Omega_1,\Omega_2,\Omega_+$) of the form
	\begin{align*}
	\Omega_-&=\{(t,\vx):t>0\text{ and }x_2<\mu_0\,t\}; \\
	\Omega_i\ &=\{(t,\vx):t>0\text{ and }\mu_{i-1}\,t<x_2<\mu_i\,t\}; \\
	\Omega_+&=\{(t,\vx):t>0\text{ and }x_2>\mu_N\,t\}.
	\end{align*}
\end{definition}

\begin{definition}
	An admissible weak solution $(\vr,\vv,p)$ of the full Euler system \eqref{eq:euler}-\eqref{eq:initial} with Riemann initial data \eqref{eq:Riemann} is called an $N$-fan solution, if there exists $N \in \{1,2\}$ and an $N$-fan partition of $(0,\infty)\times\R^2$ such that
	\begin{itemize}
		\item $(\vr,p)$ are constant on each set $\Omega_-$,$\Omega_i$, $\Omega_+$ of the $N$-fan partition;
		\item $\abs{\vv}^2$ is constant on each set $\Omega_-$,$\Omega_i$, $\Omega_+$ of the $N$-fan partition;
		\item $(\vr,\vv,p) = (\vr_\pm,\vv_\pm,p_\pm)$ on the sets $\Omega_\pm$.
	\end{itemize}
	See also Fig. \ref{fig}.
\end{definition}

\begin{figure}[hbt]
	\subfloat[1-fan solution\label{figleft}]{
		\centering
		\includegraphics[width=0.47\textwidth]{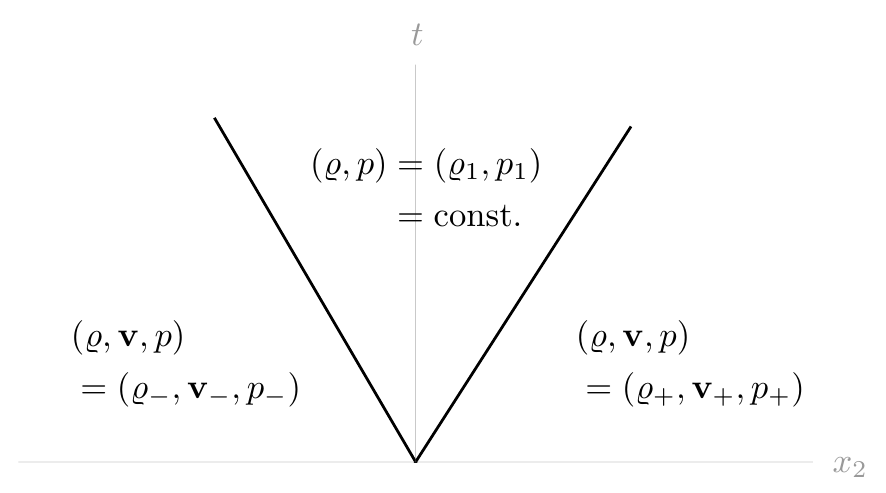}
	}
	\hfill
	\subfloat[2-fan solution\label{figright}]{
		\centering
		\includegraphics[width=0.47\textwidth]{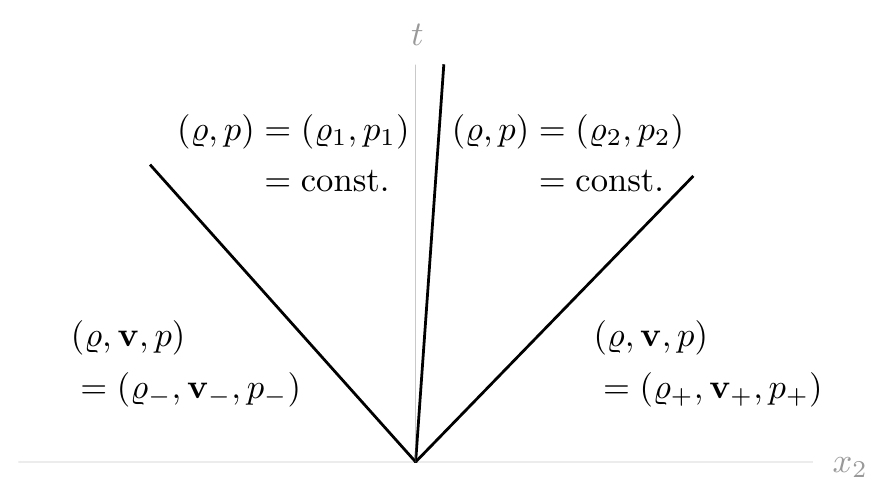}
	}
	\caption{Structure of an $N$-fan solution. On the leftmost and rightmost sets (i. e. on $\Omega_\pm$), the solution coincides with the initial states. On each set in the middle ($\Omega_i$, $i\in\{1,N\}$), the density $\vr$ and pressure $p$ are constant, whereas the velocity $\vv$ does not need to be constant.}
	\label{fig}
\end{figure}

Now we are ready to state our main theorems.

\begin{theorem} \label{thm:mainWUE}
	Let $c_v > \frac 12$ and let the Riemann initial data \eqref{eq:Riemann} be such that the 1D Riemann solution consists either 
	\begin{itemize}
		\item of a 1-shock and a 3-shock or 
		\item of a 1-shock, a 2-contact discontinuity and a 3-shock.
	\end{itemize}
	Then there exist infinitely many admissible weak solutions to \eqref{eq:euler}-\eqref{eq:Riemann}. These solutions are all $2$-fan solutions.
\end{theorem}

\begin{theorem}\label{t:main}
	Let $c_v > 0$. Let $\vr_{\pm}, p_{\pm} > 0$ and $V \in \mathbb{R}$ be given.
	\begin{itemize}
		\item[(i)] If $\vr_- < \vr_+$ or $\vr_- = \vr_+$ and $p_- > p_+$ let $v_{-,2} = V$. Then there exists $U = U(\vr_\pm,p_\pm,V)$ such that for all $v_{+,2} < U$ there exist infinitely many admissible weak solutions to \eqref{eq:euler}-\eqref{eq:Riemann}. These solutions are all $1$-fan solutions.
		\item[(ii)] If $\vr_- > \vr_+$ or $\vr_- = \vr_+$ and $p_- < p_+$ let $v_{+,2} = V$. Then there exists $U = U(\vr_\pm,p_\pm,V)$ such that for all $v_{-,2} > U$ there exist infinitely many admissible weak solutions to \eqref{eq:euler}-\eqref{eq:Riemann}. These solutions are all $1$-fan solutions.
	\end{itemize}
\end{theorem}

\begin{remark}
	Theorem \ref{thm:mainWUE} is the most general result presented here, covering all the cases of Riemann initial data giving rise to the self-similar solution containing two shocks. Solutions constructed in the proof of this Theorem are all $2$-fan solutions. It might be of interest that at least if the difference between the velocities $v_{-,2}$ and $v_{+,2}$ is large enough, one can construct also another type of solutions, namely $1$-fan solutions, this is stated in Theorem \ref{t:main}. We emphasize that the set of initial data in Theorem \ref{t:main} is a strict subset of the set of initial data in Theorem \ref{thm:mainWUE}. Finally, we refer the reader to Remark \ref{r:eqpress} for the reason why Theorem \ref{t:main} does not cover the case $\vr_- = \vr_+$ and $p_- = p_+$.
\end{remark}

\begin{remark}
	Note that the general Theorem \ref{thm:mainWUE} requires the specific heat at constant volume coefficient $c_v$ to be strictly larger than $\frac 12$, whereas the results of Theorem \ref{t:main} hold for all positive $c_v$.
\end{remark}

\section{Existence of infinitely many admissible weak solutions}\label{s:exis}

The content of this section shows the method to prove non-uniqueness of admissible weak solutions. As in the isentropic case (see \cite{ChiDelKre15} and following works) the idea is to work with a so called \emph{admissible fan subsolution} - here a quintuple of piecewise constant functions satisfying certain system of partial differential equations. Then, Theorem \ref{thm:condition} tells us that the existence of such an admissible fan subsolution implies existence of infinitely many weak entropy solutions to \eqref{eq:euler}-\eqref{eq:Riemann}.

In what follows we denote by $\Sz$ the space of symmetric traceless $2\times 2$ matrices and $\id$ denotes the $2\times 2$ identity matrix.

\begin{definition} 
	An admissible $N$-fan subsolution to the Euler system \eqref{eq:euler}-\eqref{eq:entropy} with Riemann initial data \eqref{eq:Riemann} is a quintuple $(\ov{\vr},\ov{\vv},\ov{U},\ov{C},\ov{p}):(0,\infty)\times\R^2\rightarrow(\R^+\times\R^2\times\Sz\times\R^+\times\R^+)$ of piecewise constant functions, which satisfies the following properties:
	\begin{enumerate}
		\item There exists an $N$-fan partition $\Omega_-,\Omega_1,\Omega_+$ (or $\Omega_-,\Omega_1,\Omega_2,\Omega_+$) of $(0,\infty)\times\R^2$ and for $i\in\{1,N\}$ there exist constants $\vr_i\in\R^+$, $\vv_i\in\R^2$, $U_i\in\Sz$, $C_i\in\R^+$ and $p_i\in\R^+$, such that
		\begin{align*}
		(\ov{\vr},\ov{\vv},\ov{U},\ov{C},\ov{p})&=\sum\limits_{i\in\{-,+\}} \bigg(\vr_i\,,\,\vv_i\,,\,\vv_i\otimes \vv_i - \frac{|\vv_i|^2}{2}\,\id\,,\,|\vv_i|^2\,,\,p_i\bigg)\,\mathbf{1}_{\Omega_i} + \sum\limits_{i=1}^N (\vr_i\,,\,\vv_i\,,\,U_i\,,\,C_i\,,\,p_i)\,\mathbf{1}_{\Omega_i},
		\end{align*}
		where $\vr_{\pm},\vv_{\pm},p_\pm$ are the constants given by the initial condition \eqref{eq:Riemann}. 
		\item For $i\in\{1,N\}$ the following inequality holds in the sense of definiteness:
		\begin{equation*}
		\vv_i\otimes \vv_i-U_i < \frac{C_i}{2}\,\id.
		\end{equation*}
		\item For all test functions $(\psi,\bfphi,\phi)\in \DC([0,\infty)\times\R^2,\R\times\R^2\times\R)$ the following identities hold:
		\begin{align*} 
		\intt{\intx{\big[\ov{\vr}\,\partial_t\psi + \ov{\vr}\,\ov{\vv}\cdot\Grad\psi\big]}} + \intx{ \vr^0\,\psi(0,\cdot)}\ &=\ 0, \\
		\int_0^\infty\int_{\R^2}\bigg[\ov{\vr}\,\ov{\vv}\cdot\partial_t\bfphi + \ov{\vr}\,\ov{U}:\Grad\bfphi + \bigg(\ov{p} + \half\,\ov{\vr}\,\ov{C} \bigg)\,\Div\bfphi\bigg]\dx\dt  + \intx{ \vr^0\,\vv^0\cdot\bfphi(0,\cdot)}\  &=\ 0, \\
		\int_0^\infty\int_{\R^2}\bigg[\bigg(\half\,\ov{\vr}\,\ov{C} + c_v\,\ov{p}\bigg)\,\partial_t\phi +\bigg(\half\,\ov{\vr}\,\ov{C} + (c_v+1)\,\ov{p}\bigg)\,\ov{\vv}\cdot\Grad\phi\bigg]\dx\dt & \\
		+ \intx{\bigg(\vr^0\,\frac{|\vv^0|^2}{2} + c_v\,p^0\bigg)\,\phi(0,\cdot)} \ &=\ 0. 
		\end{align*}
		\item For every non-negative test function $\varphi\in \DC([0,\infty)\times\R^2,\R_0^+)$ the inequality
		\begin{align*}
		&\intt{\intx{\Big[\ov{\vr}\,s(\ov{\vr},\ov{p})\,\partial_t\varphi + \ov{\vr}\,s(\ov{\vr},\ov{p})\,\ov{\vv}\cdot\Grad\varphi\Big]}} + \intx{ \vr^0\,s(\vr^0,p^0)\,\varphi(0,\cdot)} \leq 0 
		\end{align*}
		is fulfilled.
	\end{enumerate}
	\label{defn:fansubs}
\end{definition}

\subsection{Sufficient condition for non-uniqueness}

The key tool in the proof of Theorems \ref{thm:mainWUE} and \ref{t:main} is the following theorem relating existence of infinitely many admissible $N$-fan solutions to the existence of a single admissible $N$-fan subsolution

\begin{theorem}
	Let $(\vr_\pm,\vv_\pm,p_\pm)$ be such that there exists an admissible $N$-fan subsolution $(\ov{\vr},\ov{\vv},\ov{U},\ov{C},\ov{p})$ to the Cauchy problem \eqref{eq:euler}-\eqref{eq:Riemann}. Then there are infinitely many admissible weak $N$-fan solutions $(\vr,\vv,p)$ to \eqref{eq:euler}-\eqref{eq:Riemann} with the following properties:
	\begin{itemize}
		\item $(\vr,p)=(\ov{\vr},\ov{p})$,
		\item $\vv= \vv_\pm$ on $\Omega_\pm$, 
		\item $|\vv|^2= C_i$ a.e. on $\Omega_i$, $i=1,N$.
	\end{itemize}
	\label{thm:condition}
\end{theorem}

The proof of Theorem \ref{thm:condition} is based on a convex integration technique for the pressureless incompressible Euler equations which was introduced in seminal works of De Lellis and Sz\'ekelyhidi \cite{DLSz1,DLSz2} and is summarized in the following Proposition. 

\begin{propo} 
	Let $(\til{\vv},\til{U})\in\R^2\times \Sz$ and $C>0$ such that  $\til{\vv}\otimes\til{\vv}-\til{U}<\frac{C}{2}\,\id$. Furthermore let $\Omega\subset\R\times\R^2$ open. Then there exist infinitely many maps $(\un{\vv},\un{U})\in L^\infty(\R\times\R^2,\R^2\times\Sz)$ with the following properties.
	\begin{enumerate}
		\item $\un{\vv}$ and $\un{U}$ vanish outside $\Omega$.
		\item For all test functions $(\psi,\bfphi)\in \DC(\R\times\R^2,\R\times\R^2)$ it holds that 
		\begin{align*}
		&\iint_\Omega\un{\vv}\cdot\Grad \psi\,\dx\,\dt\ =\ 0, \\
		&\iint_\Omega(\un{\vv}\cdot\partial_t \bfphi + \un{U}:\Grad \bfphi)\dx\dt\ =\ 0. 
		\end{align*}
		\item $(\til{\vv}+\un{\vv})\otimes (\til{\vv}+\un{\vv}) - (\til{U}+\un{U}) = \frac{C}{2}\,\id$ is fulfilled almost everywhere in $\Omega$. 
	\end{enumerate}
	\label{prop:convint}
\end{propo}

For the proof of Proposition \ref{prop:convint} we refer to \cite[Lemma 3.7]{ChiDelKre15}.

The proof of Theorem \ref{thm:condition} using Proposition \ref{prop:convint} is now quite straightforward.
\begin{proof} 
	Let $(\ov{\vr},\ov{\vv},\ov{U},\ov{C},\ov{p})$ an admissible $N$-fan subsolution. Furthermore for $i=1,N$ let $(\un{\vv}_i,\un{U}_i)\in L^\infty(\R\times\R^2,\R^2\times\Sz)$ pairs of functions as in Proposition \ref{prop:convint} with $\til{\vv}=\vv_i$, $\til{U}=U_i$, $C=C_i$ and $\Omega=\Omega_i$. It suffices to prove that $(\vr,\vv,p) = \Big(\ov{\vr},\ov{\vv}+\sum\limits_{i=1}^N\un{\vv}_i,\ov{p}\Big)$ is an admissible weak solution to \eqref{eq:euler}-\eqref{eq:Riemann}. The fact that it is indeed a $N$-fan solution is clear.
	
	Let $(\psi,\bfphi,\phi)\in \DC([0,\infty)\times\R^2,\R\times\R^2\times\R)$ be test functions. Then we get
	\begin{align*}
	&\intt{\intx{\big[\vr\,\partial_t\psi + \vr\,\vv\cdot\Grad\psi\big]}} + \intx{ \vr^0\,\psi(0,\cdot)} \\
	&= \intt{\intx{\bigg[\ov{\vr}\,\partial_t\psi + \ov{\vr}\,\bigg(\ov{\vv} + \sum\limits_{i=1}^N\un{\vv}_i\bigg)\cdot\Grad\psi\bigg]}} + \intx{ \vr^0\,\psi(0,\cdot)} \\
	&= \sum\limits_{i=1}^N\iint_{\Omega_i}\vr_i\,\un{\vv}_i\cdot\Grad\psi\dx\dt \quad=\quad 0,
	\end{align*} 
	
	\begin{align*}
	&\intt{\intx{\Big[\vr\,\vv\cdot\partial_t\bfphi + \vr\,\vv\otimes \vv : \Grad\bfphi + p \,\Div\bfphi\Big]}}  + \intx{ \vr^0\,\vv^0\cdot\bfphi(0,\cdot)} \\
	&=\int_0^\infty\int_{\R^2}\bigg[\ov{\vr}\,\ov{\vv}\cdot\partial_t\bfphi + \ov{\vr}\,\ov{U}:\Grad\bfphi + \bigg(\ov{p} + \half\,\ov{\vr}\,\ov{C} \bigg)\,\Div\bfphi\bigg]\dx\dt  + \intx{ \vr^0\,\vv^0 \cdot \bfphi(0,\cdot)} \\
	&\quad + \sum\limits_{i=1}^N \iint_{\Omega_i}\Big[\vr_i\,\un{\vv}_i\cdot\partial_t\bfphi + \vr_i\,\un{U}_i : \Grad\bfphi\Big]\dx\dt  \quad =\quad 0,
	\end{align*} 
	
	\begin{align*}
	&\intt{\intx{\bigg[\bigg(\half\vr|\vv|^2 + c_v\,p\bigg)\,\partial_t\phi +\bigg(\half\vr|\vv|^2 + (c_v+1)\,p\bigg)\,\vv\cdot\Grad\phi\bigg]}}  \\
	&\quad + \intx{\bigg(\vr^0\,\frac{|\vv^0|^2}{2} + c_v\,p^0\bigg)\,\phi(0,\cdot)} \\
	&=\int_0^\infty\int_{\R^2}\bigg[\bigg(\half\,\ov{\vr}\,\ov{C} + c_v\,\ov{p}\bigg)\,\partial_t\phi + \bigg(\half\,\ov{\vr}\,\ov{C} + (c_v+1)\,\ov{p}\bigg)\,\ov{\vv}\cdot\Grad\phi\bigg]\dx\dt \\
	&\qquad + \intx{ \bigg(\vr^0\,\frac{|\vv^0|^2}{2} + c_v\,p^0\bigg)\,\phi(0,\cdot)} + \sum\limits_{i=1}^N \iint_{\Omega_i} \bigg(\vr_i\,\frac{C_i}{2} + (c_v+1)\,p_i\bigg)\,\un{\vv}_i\cdot\Grad\phi\dx\dt  \quad = \quad 0.
	\end{align*}
	
	For every non-negative test function $\varphi\in \DC([0,\infty)\times\R^2,\R_0^+)$ it holds
	\begin{align*}
	&\intt{\intx{\bigg[\vr\,s(\vr,p)\,\partial_t\varphi + \vr\,s(\vr,p)\,\vv\cdot\Grad\varphi\bigg]}} + \intx{ \vr^0\,s(\vr^0,p^0)\,\varphi(0,\cdot)}  \\
	&= \intt{\intx{\Big[\ov{\vr}\,s(\ov{\vr},\ov{p})\,\partial_t\varphi + \ov{\vr}\,s(\ov{\vr},\ov{p})\,\ov{\vv}\cdot\Grad\varphi\Big]}} + \intx{\vr^0\,s(\vr^0,p^0)\,\varphi(0,\cdot)} \\
	&\qquad + \sum\limits_{i=1}^N \iint_{\Omega_i}\vr_i\,s(\vr_i,p_i)\,\un{\vv}_i\cdot\Grad\varphi\dx\dt \quad \leq\quad  0.
	\end{align*}
\end{proof}

\subsection{Algebraic equations and inequalities} 

Since the admissible $N$-fan subsolution consists of a quintuple of piecewise constant functions which satisfy a certain set of partial differential equations and inequalities, it is easy to observe that these differential constraints are equivalent to a system of algebraic equations and inequalities which arise from the Rankine-Hugoniot conditions on the interfaces between the sets of the $N$-fan partition. More precisely we have the following.

\begin{propo} 
	Let $\vr_\pm, p_\pm\in\R^+$, $\vv_\pm\in\R^2$ be given. The constants $\mu_0,\mu_1\in\R$ (resp. $\mu_0,\mu_1,\mu_2\in\R$) and $\vr_i,p_i\in\R^+$,
	\begin{align*}
	\vv_i&=\left(\begin{array}{c}
	\alpha_i \\
	\beta_i
	\end{array}\right)\in\R^2, & U_i&=\left(\begin{array}{rr}
	\gamma_i & \delta_i \\
	\delta_i & -\gamma_i
	\end{array}\right)\in\Sz, 
	\end{align*}
	and $C_i\in\R^+$ (for $i=1,N$) define an admissible $N$-fan subsolution to the Cauchy problem \eqref{eq:euler}-\eqref{eq:Riemann} if and only if they fulfill the following algebraic equations and inequalities: 
	\begin{itemize}
		\item Order of the speeds:
		\begin{align}
		\mu_0&<\mu_1 & &\text{resp.} & \mu_0&<\mu_1<\mu_2
		\label{eq:order}
		\end{align}
		\item Rankine Hugoniot conditions on the left interface:
		\begin{align}
		\mu_0\,(\vr_- - \vr_1) &= \vr_-\,v_{-,2} - \vr_1\,\beta_1 \label{eq:rhl1}\\
		\mu_0\,(\vr_-\,v_{-,1} - \vr_1\,\alpha_1) &= \vr_-\,v_{-,1}\,v_{-,2} - \vr_1\,\delta_1 \label{eq:rhl2}\\
		\mu_0\,(\vr_-\,v_{-,2} - \vr_1\,\beta_1) &= \vr_-\,v_{-,2}^2 - \vr_1\,\bigg(\frac{C_1}{2}-\gamma_1\bigg) + p_- - p_1  \label{eq:rhl3} 
		\end{align}
		\begin{equation}
		\begin{split}
		&\mu_0\,\bigg(\half\vr_-\,|v_-|^2 + c_v\,p_- - \vr_1\,\frac{C_1}{2} - c_v\,p_1\bigg) = \\
		&\quad\bigg(\half\vr_-\,|v_-|^2 + (c_v+1)\,p_-\bigg)\,v_{-,2} - \bigg(\vr_1\,\frac{C_1}{2} + (c_v+1)\,p_1\bigg)\,\beta_1
		\end{split}
		\label{eq:rhl4}
		\end{equation}
		
		\item Rankine Hugoniot conditions on the right interface:
		\begin{align}
		\mu_N\,(\vr_N - \vr_+) &= \vr_N\,\beta_N - \vr_+\,v_{+,2} \label{eq:rhr1}\\
		\mu_N\,(\vr_N\,\alpha_N - \vr_+\,v_{+,1}) &= \vr_N\,\delta_N - \vr_+\,v_{+,1}\,v_{+,2} \label{eq:rhr2}\\
		\mu_N\,(\vr_N\,\beta_N - \vr_+\,v_{+,2}) &= \vr_N\,\bigg(\frac{C_N}{2}-\gamma_N\bigg) - \vr_+\,v_{+,2}^2 + p_N - p_+ \label{eq:rhr3}
		\end{align}
		\begin{equation}
		\begin{split}
		&\mu_N\,\bigg(\vr_N\,\frac{C_N}{2} + c_v\,p_N - \half\vr_+\,|v_+|^2 - c_v\,p_+\bigg) = \\
		&\quad\bigg(\vr_N\,\frac{C_N}{2} + (c_v+1)\,p_N\bigg)\,\beta_N - \bigg(\half\vr_+\,|v_+|^2 + (c_v+1)\,p_+\bigg)\,v_{+,2}
		\end{split}
		\label{eq:rhr4}
		\end{equation}
		
		\item If $N=2$ we additionally have Rankine Hugoniot conditions on the middle interface: 
		\begin{align}
		\mu_1\,(\vr_1 - \vr_2) &= \vr_1\,\beta_1 - \vr_2\,\beta_2 \label{eq:rhm1}\\
		\mu_1\,(\vr_1\,\alpha_1 - \vr_2\,\alpha_2) &= \vr_1\,\delta_1 - \vr_2\,\delta_2 \label{eq:rhm2}\\
		\mu_1\,(\vr_1\,\beta_1 - \vr_2\,\beta_2) &= \vr_1\,\bigg(\frac{C_1}{2}-\gamma_1\bigg) - \vr_2\,\bigg(\frac{C_2}{2}-\gamma_2\bigg) + p_1 - p_2  \label{eq:rhm3}
		\end{align}
		\begin{equation}
		\begin{split}
		&\mu_1\,\bigg(\vr_1\,\frac{C_1}{2} + c_v\,p_1 - \vr_2\,\frac{C_2}{2} - c_v\,p_2\bigg) = \\
		&\quad\bigg(\vr_1\,\frac{C_1}{2} + (c_v+1)\,p_1\bigg)\,\beta_1 - \bigg(\vr_2\,\frac{C_2}{2} + (c_v+1)\,p_2\bigg)\,\beta_2
		\end{split}
		\label{eq:rhm4}
		\end{equation}
		
		\item Subsolution conditions for $i=1,N$:
		\begin{align}
		C_i - \alpha_i^2 - \beta_i^2 &> 0 \label{eq:sc1}\\
		\bigg(\frac{C_i}{2} - \alpha_i^2 + \gamma_i\bigg) \bigg(\frac{C_i}{2} - \beta_i^2 - \gamma_i\bigg) - (\delta_i-\alpha_i\,\beta_i)^2 &> 0 \label{eq:sc2} 
		\end{align}
		\item Admissibility condition on the left interface:
		\begin{equation}
		\mu_0\,\Big(\vr_1\,s(\vr_1,p_1) - \vr_-\,s(\vr_-,p_-)\Big)\leq \vr_1\,s(\vr_1,p_1)\,\beta_1 - \vr_-\,s(\vr_-,p_-)\,v_{-,2}
		\label{eq:adml}
		\end{equation}
		\item Admissibility condition on the right interface: 
		\begin{equation}
		\mu_N\,\Big(\vr_+\,s(\vr_+,p_+) - \vr_N\,s(\vr_N,p_N)\Big) \leq \vr_+\,s(\vr_+,p_+)\,v_{+,2} - \vr_N\,s(\vr_N,p_N)\,\beta_N
		\label{eq:admr}
		\end{equation}
		\item If $N=2$ we additionally have an admissibility condition on the middle interface: 
		\begin{equation}
		\mu_1\,\Big(\vr_2\,s(\vr_2,p_2) - \vr_1\,s(\vr_1,p_1)\Big) \leq \vr_2\,s(\vr_2,p_2)\,\beta_2 - \vr_1\,s(\vr_1,p_1)\,\beta_1
		\label{eq:admm}
		\end{equation}
	\end{itemize}
	\label{prop:algequa}
\end{propo} 

\begin{proof}
	There is really nothing to prove in this proposition, differential constraints for piecewise constant functions translate to Rankine-Hugoniot conditions in a standard way and the subsolution conditions \eqref{eq:sc1}-\eqref{eq:sc2} state that the matrix $\frac{C_i}{2}\id - \vv_i\otimes\vv_i + U_i$ is positive definite, which holds if and only if its trace and its determinant are both positive.
\end{proof}

As we already mentioned, in order to prove our main theorems, due to Theorem \ref{thm:condition} it is always enough to find a single $N$-fan subsolution. Moreover, we are going to make a special ansatz for the subsolutions. Recall that we assume throughout this paper that $v_{-,1} = v_{+,1} = 0$. We will look for $N$-fan subsolutions satisfying
\begin{equation}\label{eq:choice}
\alpha_i = 0 \qquad \& \qquad \delta_i = \alpha_i\beta_i = 0,
\end{equation}
$i = 1,N$. One can easily check that this choice implies that equations \eqref{eq:rhl2}, \eqref{eq:rhr2} and \eqref{eq:rhm2} are trivially satisfied. We also note that in the case $N = 1$ this is in fact a consequence of the set of Rankine-Hugoniot conditions, for more details see \cite[Lemma 4.2]{ChiKre14}.

For simplicity of notation, we will use in the rest of the paper notation $v_\pm$ in place of $v_{\pm,2}$ and similarly $v_M$ instead of $v_{M,2}$.

\section{Proof of Theorem \ref{thm:mainWUE}} \label{s:proofB}

\subsection{Reduction to special case}

We recall that the proof of Theorem \ref{thm:mainWUE} is finished as soon as we find a single admissible $2$-fan subsolution. Since the structure of the $2$-fan subsolution is similar to the 1D Riemann solution to the problem, we will look for the $2$-fan subsolution as a suitable small perturbation of the 1D Riemann solution.

Moreover, we claim that we can make the following choice without loss of generality.

\begin{propo} \label{prop:mainWUE}
	Let the assumption of Theorem \ref{thm:mainWUE} hold and assume furthermore that $v_M = 0$, where $v_M$ is the second component of the velocity of the 1D Riemann solution (see Proposition \ref{prop:standard}). Then there exist infinitely many admissible weak $2$-fan solutions to the Riemann problem \eqref{eq:euler}-\eqref{eq:Riemann}.
\end{propo}

Proposition \ref{prop:mainWUE} implies Theorem \ref{thm:mainWUE} using the Galilean invariance of the Euler equations. Indeed, if we consider new initial data as follows:
\begin{equation} \label{eq:newinitial}
(\vr^{0,\new},\vv^{0,\new},p^{0,\new}):=(\vr^0,\vv^0 - (0,v_M)^T,p^0).
\end{equation}
the 1D Riemann solution to the problem \eqref{eq:euler}-\eqref{eq:entropy} with \eqref{eq:newinitial} will fulfill $v_M^\new = 0$. We apply Proposition \ref{prop:mainWUE} and find infinitely many admissible weak $2$-fan solutions $(\vr,\vv,p)$ to the Riemann problem with initial data \eqref{eq:newinitial}. Then, again by Galilean invariance, $$\vr(t, \vx-(0,v_M)^Tt),\ \vv(t,\vx-(0,v_M)^Tt) + (0,v_M)^T,\ p(t, \vx -(0,v_
M)^Tt)$$ are admissible weak $2$-fan solutions to the Riemann problem with original initial data \eqref{eq:Riemann}, and the proof of Theorem \ref{thm:mainWUE} is finished.

What remains now is to prove Proposition \ref{prop:mainWUE}.

\subsection{Proof of Proposition \ref{prop:mainWUE}}	

Let the assumptions of Proposition \ref{prop:mainWUE} be true. As we mentioned above, we will look for an admissible $2$-fan subsolution as a suitable perturbation of the 1D Riemann solution. This perturbation will be quantified by a small parameter $\ep > 0$.

For convenience we define functions $A,B,D:\R\rightarrow\R$ by 
\begin{align*} 
A(\ep)&:= \vr_-(\vr_{M-}+\ep)(\vr_{M+}-\ep - \vr_+) - \vr_+(\vr_{M+}-\ep)(\vr_{M-}+\ep - \vr_-); \\
B(\ep)&:= \vr_-\vr_+(\vr_{M-}+\ep)(\vr_{M+}-\ep)\big(v_{-} - v_{+}\big)^2 - (p_- - p_+)\ A(\ep); \\
D(\ep) &:= v_{-} \vr_- (\vr_{M-} + \ep) (\vr_{M+} - \ep - \vr_+) - v_{+} \vr_+ (\vr_{M+}-\ep) (\vr_{M-} + \ep - \vr_-).
\end{align*}

First of all we want to show some properties of the functions $A$ and $B$. It is easy to deduce from Proposition \ref{prop:standard} and the assumption $v_M=0$, that 
\begin{align}
v_- &= \sqrt{2\,c_v}\frac{p_M - p_-}{\sqrt{\vr_- (p_- + (2c_v + 1) p_M)}} ,\label{eq:1} \\
v_+ &= 	-\sqrt{2\,c_v}\frac{p_M - p_+}{\sqrt{\vr_+ (p_+ + (2c_v + 1) p_M)}}. \label{eq:2}
\end{align}
Since $p_M > \max\{p_-,p_+\}$, we have $v_+<0<v_-$. Furthermore we obtain from Proposition \ref{prop:standard} and \eqref{eq:1}, \eqref{eq:2} that 
\begin{align*}
\frac{\vr_{M-} - \vr_-}{\vr_-\vr_{M-}} = \frac{2c_v (p_M - p_-)}{\vr_- (p_- + (2c_v + 1) p_M)} = \frac{\big(v_-\big)^2}{p_M - p_-}, \\
\frac{\vr_{M+} - \vr_+}{\vr_+\vr_{M+}} = \frac{2c_v (p_M - p_+)}{\vr_+ (p_+ + (2c_v + 1) p_M)} = \frac{\big(v_+\big)^2}{p_M - p_+}. 
\end{align*}
This leads to 
\begin{align*}
\big(v_-\big)^2 + (p_- - p_+)\frac{\vr_{M-} - \vr_-}{\vr_-\vr_{M-}} &= \big(v_-\big)^2 \frac{p_M - p_+}{p_M - p_-}, \\ \\
\big(v_+\big)^2 - (p_- - p_+)\frac{\vr_{M+} - \vr_+}{\vr_+\vr_{M+}} &= \big(v_+\big)^2 \frac{p_M - p_-}{p_M - p_+},
\end{align*}
and finally
\begin{align*}
&B(0) = \notag\\
& = \vr_-\vr_+\vr_{M-}\vr_{M+}\big(v_- - v_+\big)^2 - (p_- - p_+)\Big(\vr_-\vr_{M-}\big(\vr_{M+} - \vr_+\big) - \vr_+\vr_{M+}\big(\vr_{M-} - \vr_-\big)\Big) \notag\\
& = \vr_-\vr_+\vr_{M-}\vr_{M+}\bigg[\big(v_-\big)^2 + (p_- - p_+)\frac{\vr_{M-} - \vr_-}{\vr_-\vr_{M-}} + \big(v_+\big)^2 - (p_- - p_+)\frac{\vr_{M+} - \vr_+}{\vr_+\vr_{M+}} -2v_-v_+ \bigg] \notag\\
& = \vr_-\vr_+\vr_{M-}\vr_{M+}\bigg[\big(v_-\big)^2 \frac{p_M - p_+}{p_M - p_-} + \big(v_+\big)^2 \frac{p_M - p_-}{p_M - p_+} -2v_-v_+ \bigg] \quad >\quad 0. 
\end{align*} 
Now, by continuity of the function $B$, there exists an $\ep_{\max,1}>0$ such that $B(\ep)>0$ for all $\ep\in(0,\ep_{\max,1}]$. Because $\vr_{M+}>\vr_+$, there exists $\ep_{\max,2}>0$ such that $\vr_{M+}-\ep-\vr_+ > 0$ for all $\ep\in(0,\ep_{\max,2}]$.

Next we want to show that there is an $\ep_{\max,3}>0$ such that $A(\ep)\neq 0$ for all $\ep\in(0,\ep_{\max,3}]$. To this end, let us first assume, that $A(0)\neq 0$. Then, by continuity of the function $A$, there exists such an $\ep_{\max,3}>0$. Now consider the case where $A(0) = 0$. Then we obtain 
\begin{align*}
A(\ep) &=  \vr_-(\vr_{M-}+\ep)(\vr_{M+}-\ep - \vr_+) - \vr_+(\vr_{M+}-\ep)(\vr_{M-}+\ep - \vr_-) \\
&= \ep^2 (\vr_+ - \vr_-) - \ep \big(2\vr_-\vr_+ + (\vr_- - \vr_+)(\vr_{M-} - \vr_{M+})\big) + \underbrace{A(0)}_{=0} \\
&= \ep \big((\vr_+ - \vr_-)\ep -2\vr_-\vr_+ - (\vr_- - \vr_+)(\vr_{M-} - \vr_{M+})\big),
\end{align*}
which has at most two zeros: If $\vr_- = \vr_+$ then 
\begin{align*}
A(\ep)=0 \qquad \Longleftrightarrow \qquad \ep=0;
\end{align*}
and if $\vr_- \neq \vr_+$ then 
\begin{align*}
A(\ep)=0 \qquad \Longleftrightarrow \qquad \ep=0\ \text{ or }\ \ep=\frac{2\vr_-\vr_+ + (\vr_- - \vr_+)(\vr_{M-} - \vr_{M+})}{\vr_+ - \vr_-}.
\end{align*}
Hence there exists $\ep_{\max,3}>0$ such that $A(\ep)\neq 0$ for all $\ep\in(0,\ep_{\max,3}]$.

We set $\ep_{\max}:=\min\{\ep_{\max,1},\ep_{\max,2},\ep_{\max,3}\}$ and then we have
\begin{equation} \label{eq:28}
\begin{split}
A(\ep) &\neq 0, \qquad\qquad \vr_{M+}-\ep-\vr_+ > 0, \\
B(\ep) &> 0, \qquad\qquad \vr_{M-} + \ep-\vr_- > 0,
\end{split}
\end{equation}
for all $\ep\in(0,\ep_{\max}]$.

\subsubsection{Shock speeds}

Next we define the functions $\mu_0,\mu_1,\mu_2:(0,\ep_{\max}]\rightarrow\R$ by 
\begin{align*} 
\mu_0(\ep)&:=\frac{1}{A(\ep)} \bigg[D(\ep) +\vr_-\vr_+(\vr_{M+} - \ep) (v_- - v_+)  - \sqrt{\big(\vr_{M-} + \ep\big)^2\ \frac{\vr_{M+} - \ep - \vr_+}{\vr_{M-} + \ep-\vr_-}\ B(\ep)} \ \bigg]; \\
\mu_1(\ep)&:=\frac{1}{A(\ep)} \bigg[D(\ep)- \sqrt{(\vr_{M-} + \ep-\vr_-)(\vr_{M+} - \ep - \vr_+)\ B(\ep)}\ \bigg]; \\
\mu_2(\ep)&:= \frac{1}{A(\ep)} \bigg[D(\ep) + \vr_-\vr_+(\vr_{M-} + \ep) (v_- - v_+) - \sqrt{(\vr_{M+} - \ep)^2\ \frac{\vr_{M-} + \ep - \vr_-}{\vr_{M+} - \ep - \vr_+}\ B(\ep)}\ \bigg].
\end{align*}
Note first that the functions $\mu_0,\mu_1,\mu_2$ are well-defined because of \eqref{eq:28}. We claim that these functions define perturbations of the shock speeds $\sigma_-$, $v_M$ and $\sigma_+$ of the 1D Riemann solution. More precisely we have

\begin{propo}
	It holds that 
	\begin{align*}
	\lim\limits_{\ep\to 0}\mu_0(\ep) &= \sigma_-, &
	\lim\limits_{\ep\to 0}\mu_1(\ep) &= v_M, &
	\lim\limits_{\ep\to 0}\mu_2(\ep) &= \sigma_+. 
	\end{align*}
	\label{prop:limittostandard1}
\end{propo}

\begin{proof}
	We start with the Rankine-Hugoniot conditions for the 1D Riemann solution 
	\begin{align}
	\sigma_-\,(\vr_- - \vr_{M-}) &= \vr_-\,v_- - \vr_ {M-}\,v_M; \label{eq:rhstandardl1}\\
	\sigma_-\,(\vr_-\,v_- - \vr_{M-}\,v_M) &= \vr_-\,v_-^2 - \vr_{M-}\,v_M^2 + p_- - p_M;\label{eq:rhstandardl2}\\
	\sigma_+\,(\vr_{M+} - \vr_+) &= \vr_ {M+}\,v_M - \vr_+\,v_+; \label{eq:rhstandardr1}\\
	\sigma_+\,(\vr_{M+}\,v_M - \vr_+\,v_+) &= \vr_{M+}\,v_M^2 - \vr_+\,v_+^2 + p_M - p_+ \label{eq:rhstandardr2}
	\end{align}
	and we obtain by eliminating $\sigma_-$, $\sigma_+$ and $p_M$ that
	\begin{align*}
	&(\vr_{M+}-\vr_+) \big(\vr_-\,v_- - \vr_{M-}\,v_M\big)^2 + (\vr_- - \vr_{M-}) \big(\vr_{M+}\,v_M - \vr_+\,v_+\big)^2 \\
	&= \Big(\vr_-\,v_-^2 - \vr_+\,v_+^2 - \big(v_M\big)^2 (\vr_{M-} - \vr_{M+}) + p_- - p_+\Big) (\vr_{M+} - \vr_+) (\vr_- - \vr_{M-}) . 
	\end{align*}
	In order to solve this equation for $v_M$, we write it as follows
	\begin{equation} \label{eq:quadratic_vm2}
	A(0)\ \big(v_M\big)^2 - 2 D(0)\ v_M + E = 0,
	\end{equation}
	where the constant
	\begin{equation*} 
	E := (p_- - p_+)(\vr_{M-} - \vr_-)(\vr_{M+} - \vr_+) + \big(v_-\big)^2 \vr_- \vr_{M-} (\vr_{M+} - \vr_+) - \big(v_+\big)^2 \vr_+ \vr_{M+} (\vr_{M-} - \vr_-) 
	\end{equation*}
	only depends on the initial states.	Now we have to consider two cases, namely $A(0)=0$ and $A(0)\neq 0$.
	
	Let us start with $A(0)=0$. Then we easily deduce that 
	\begin{equation*}
	D(0) = \vr_- \vr_{M-} (\vr_{M+} - \vr_+) \big(v_- - v_+\big),
	\end{equation*}
	which does not vanish because $v_+<0<v_-$ and $\vr_{M+} > \vr_+$. Hence we get from \eqref{eq:quadratic_vm2}, that 
	\begin{equation} \label{eq:4}
	v_M = \frac{E}{2D(0)}.
	\end{equation}
	Next we want to compute $\lim\limits_{\ep \to 0}\mu_1(\ep)$ and compare it with \eqref{eq:4}. Keeping in mind that we are considering the case $A(0)=0$, we get 
	\begin{align*}
	&D(0)- \sqrt{(\vr_{M-} - \vr_-)(\vr_{M+} - \vr_+)\ B(0)} \\
	&=\big(v_- - v_+\big) \vr_- \vr_{M-} (\vr_{M+} - \vr_+) - \sqrt{(\vr_{M-} - \vr_-)(\vr_{M+} - \vr_+)  \vr_- \vr_+ \vr_{M-}\vr_{M+}\big(v_- - v_+\big)^2} \\
	&=\big(v_- - v_+\big) \vr_- \vr_{M-} (\vr_{M+} - \vr_+) - \sqrt{(\vr_{M+} - \vr_+)^2  (\vr_+)^2 (\vr_{M+})^2\big(v_- - v_+\big)^2}\quad =\quad 0.
	\end{align*}
	Hence we can apply L'Hospital's rule. We obtain 
	\begin{align*}
	\lim\limits_{\ep\to 0 } A'(\ep) &= \lim\limits_{\ep\to 0 }\Big(\vr_-(\vr_{M+}-\ep - \vr_+) - \vr_-(\vr_{M-}+\ep) - \vr_+(\vr_{M+}-\ep) + \vr_+(\vr_{M-}+\ep - \vr_-)\Big) \\
	&= -2\vr_-\vr_+ - (\vr_- - \vr_+) (\vr_{M-} - \vr_{M+}).
	\end{align*} 
	A short calculation shows that this is non-zero: From $A(0)=0$ we can deduce that 
	$$
	\frac{\vr_- - \vr_+}{\vr_-\vr_+} = \frac{\vr_{M-} - \vr_{M+}}{\vr_{M-}\vr_{M+}}.
	$$
	This means that $\vr_- - \vr_+$ and $\vr_{M-} - \vr_{M+}$ have the same sign, which implies $(\vr_- - \vr_+)(\vr_{M-} - \vr_{M+})\geq 0$. Since $\vr_-\vr_+ >0$, we have $A'(0)<0$, in particular $A'(0)\neq 0$
	
	Furthermore a long but straightforward computation yields
	\begin{equation*}
	\lim\limits_{\ep\to 0 }\Big[D(\ep)- \sqrt{(\vr_{M-} + \ep-\vr_-)(\vr_{M+} - \ep - \vr_+)\ B(\ep)}\ \Big]' = \frac{A'(0)\ E}{2D(0)}.
	\end{equation*} 
	Hence by L'Hospital's rule we obtain $\lim\limits_{\ep\to 0 }\mu_1(\ep)= \frac{E}{2D(0)}$ and recalling \eqref{eq:4}, we deduce $\lim\limits_{\ep\to 0 }\mu_1(\ep) = v_M$.
	
	Let us now consider the case $A(0)\neq 0$. Then we obtain from \eqref{eq:quadratic_vm2} that 
	\begin{align*}
	v_M &= \frac{1}{A(0)}\Big[D(0) \pm \sqrt{D(0)^2 - A(0)\ E\ }\ \Big].
	\end{align*}
	The correct sign in the equation above is ``$-$'' because\footnote{Alternatively, this can be verified by considering the admissibility criterion.} $v_M=0$ and $D(0)>0$, which easily follows from $v_+<0<v_-$.
	Furthermore it is simple to check that $$D(0)^2 - A(0)\ E=(\vr_{M-} - \vr_-)(\vr_{M+} - \vr_+)\ B(0).$$ 
	Then it is easy to conclude $\mu_1(0) = v_M$. 
	
	To finish the proof of Proposition \ref{prop:limittostandard1} we have to show that $\lim\limits_{\ep\to 0 } \mu_1(\ep)=v_M$ implies that $\lim\limits_{\ep\to 0 } \mu_0(\ep)=\sigma_-$ and $\lim\limits_{\ep\to 0 } \mu_2(\ep)=\sigma_+$. It is straightforward to deduce that 
	\begin{align*}
	\mu_0(\ep)&=v_- + \frac{\vr_{M-} + \ep}{\vr_{M-} + \ep - \vr_-} \big(\mu_1(\ep) - v_-\big); \\
	\mu_2(\ep)&=v_+ + \frac{\vr_{M+} - \ep}{\vr_{M+} - \ep - \vr_+} \big(\mu_1(\ep) - v_+\big). 
	\end{align*}
	On the other hand we get from \eqref{eq:rhstandardl1} and \eqref{eq:rhstandardr1}, that
	\begin{align*}
	\sigma_-&=v_- + \frac{\vr_{M-}}{\vr_{M-} - \vr_-} \big(v_M - v_-\big); &
	\sigma_+&=v_+ + \frac{\vr_{M+}}{\vr_{M+} - \vr_+} \big(v_M - v_+\big).
	\end{align*}
	Hence we easily deduce $\lim\limits_{\ep\to 0 } \mu_0(\ep)=\sigma_-$ and $\lim\limits_{\ep\to 0 } \mu_2(\ep)=\sigma_+$.
\end{proof} 

Because of $\sigma_- < v_M < \sigma_+$ and the continuity of the functions $\mu_0,\mu_1,\mu_2$, we may assume that $\mu_0(\ep) < \mu_1(\ep) < \mu_2(\ep)$ for all $\ep \in(0,\ep_{\max}]$. If this is not the case we redefine $\ep_{\max}$ to be a bit smaller that the smallest positive value of $\ep$ for which $\mu_0(\ep)<\mu_1(\ep)<\mu_2(\ep)$ is violated. 

\subsubsection{Constants $C_i$ and $\gamma_i$}

In order to proceed further we need to introduce a second positive parameter $\delta > 0$. We define the functions $C_1,C_2,\gamma_1,\gamma_2:(0,\ep_{\max}]\times(0,p_M)$ by
\begin{align*} 
C_1(\ep,\delta)&:= \frac{2}{(\vr_{M-} + \ep)\big(\mu_0(\ep) - \mu_1(\ep)\big)} \bigg[-\mu_0(\ep)\Big(c_v(p_M - \delta - p_-)-\half\vr_-|v_-|^2\Big) \\ 
&\qquad\qquad\qquad +\mu_1(\ep)(c_v+1)(p_M - \delta) - \Big(\half\vr_- |v_-|^2 + (c_v+1)p_-\Big) v_-\bigg]; \\
C_2(\ep,\delta)&:= \frac{2}{(\vr_{M+} - \ep)\big(\mu_2(\ep) - \mu_1(\ep)\big)}\bigg[-\mu_2(\ep)\Big(c_v(p_M - \delta - p_+)-\half\vr_+|v_+|^2\Big) \\ 
&\qquad\qquad\qquad +\mu_1(\ep)(c_v+1)(p_M - \delta) - \Big(\half\vr_+ |v_+|^2 + (c_v+1)p_+\Big)v_+\bigg]; \\
\gamma_1(\ep,\delta)&:= \frac{1}{(\vr_{M-} + \ep)}\bigg[(\vr_{M-} + \ep)\frac{C_1(\ep,\delta)}{2} - \vr_- v_-^2 + p_M - \delta - p_- - \mu_0(\ep) \big((\vr_{M-} + \ep)\mu_1(\ep) - \vr_- v_-\big)\bigg];\\
\gamma_2(\ep,\delta)&:= \frac{1}{(\vr_{M+} - \ep)}\bigg[(\vr_{M+} - \ep)\frac{C_2(\ep,\delta)}{2} - \vr_+ v_+^2 + p_M - \delta - p_+ - \mu_2(\ep) \big((\vr_{M+} - \ep)\mu_1(\ep) - \vr_+ v_+\big)\bigg].
\end{align*}
Note that these functions are well-defined because of the arguments above. More precisely, it holds that $\mu_0(\ep)-\mu_1(\ep)\neq 0$ and $\mu_2(\ep)-\mu_1(\ep)\neq 0$ for all $\ep \in(0,\ep_{\max}]$.

\begin{propo}
	It holds that 
	\begin{align*} 
	\lim\limits_{(\ep,\delta)\to 0} C_1(\ep,\delta) &= \big(v_M\big)^2; & \lim\limits_{(\ep,\delta)\to 0} \gamma_1(\ep,\delta) &= -\frac{\big(v_M\big)^2}{2}; \\
	\lim\limits_{(\ep,\delta)\to 0} C_2(\ep,\delta) &= \big(v_M\big)^2; & \lim\limits_{(\ep,\delta)\to 0} \gamma_2(\ep,\delta) &= -\frac{\big(v_M\big)^2}{2}.
	\end{align*}
	\label{prop:limittostandard2}
\end{propo}

\begin{proof}
	To prove this, we need the Rankine Hugoniot conditions of the 1D Riemann solution in the energy equation
	\begin{align}
	\begin{split} \label{eq:rhstandardl3}
	&\sigma_-\,\bigg(\half\vr_-\,|v_-|^2 + c_v\,p_- - \half\vr_{M-}\,|v_M|^2 - c_v\,p_M\bigg) \\
	&\quad= \bigg(\half\vr_-\,|v_-|^2 + (c_v+1)\,p_-\bigg)\,v_- - \bigg(\half\vr_{M-}\,|v_M|^2 + (c_v+1)\,p_M\bigg)\,v_M;
	\end{split} \\
	\begin{split} \label{eq:rhstandardr3}
	&\sigma_+\,\bigg(\half\vr_{M+}\,|v_M|^2 + c_v\,p_M - \half\vr_+\,|v_+|^2 - c_v\,p_+ \bigg) \\
	&\quad= \bigg(\half\vr_{M+}\,|v_M|^2 + (c_v+1)\,p_M\bigg)\,v_M - \bigg(\half\vr_+\,|v_+|^2 + (c_v+1)\,p_+\bigg)\,v_+.
	\end{split}
	\end{align}
	
	We obtain that $\lim\limits_{(\ep,\delta)\to 0} C_1(\ep,\delta) = \big(v_M\big)^2$ and $\lim\limits_{(\ep,\delta)\to 0} C_2(\ep,\delta) = \big(v_M\big)^2$ by using Proposition \ref{prop:limittostandard1} and \eqref{eq:rhstandardl3}, \eqref{eq:rhstandardr3}.
	
	The fact that $\lim\limits_{(\ep,\delta)\to 0} \gamma_1(\ep,\delta) = -\frac{(v_M)^2}{2}$ and $\lim\limits_{(\ep,\delta)\to 0} \gamma_2(\ep,\delta) = -\frac{(v_M)^2}{2}$ can be shown analogously by using the Rankine Hugoniot conditions \eqref{eq:rhstandardl2}, \eqref{eq:rhstandardr2}.
\end{proof}

We continue the proof of the Proposition \ref{prop:mainWUE} by observing that the perturbations defined above indeed help to define an admissible $2$-fan subsolution.

\begin{propo} 
	If there exists $(\ep,\delta)\in (0,\ep_{\max}]\times(0,p_M)$ such that the following inequalities are fulfilled, then there exists an admissible $2$-fan subsolution to the Cauchy problem \eqref{eq:euler}-\eqref{eq:Riemann}.
	\begin{itemize}
		\item Order of the speeds:
		\begin{equation} 
		\mu_0(\ep)<\mu_1(\ep)<\mu_2(\ep)
		\label{eq:3order}
		\end{equation} 
		
		\item Subsolution conditions:
		\begin{align}
		C_1(\ep,\delta) - \mu_1(\ep)^2 &> 0 \label{eq:3sc1}\\
		C_2(\ep,\delta) - \mu_1(\ep)^2 &> 0 \label{eq:3sc2}\\
		\bigg(\frac{C_1(\ep,\delta)}{2} + \gamma_1(\ep,\delta)\bigg) \bigg(\frac{C_1(\ep,\delta)}{2} - \mu_1(\ep)^2 - \gamma_1(\ep,\delta)\bigg)  &> 0 \label{eq:3sc3} \\
		\bigg(\frac{C_2(\ep,\delta)}{2} + \gamma_2(\ep,\delta)\bigg) \bigg(\frac{C_2(\ep,\delta)}{2} - \mu_1(\ep)^2 - \gamma_2(\ep,\delta)\bigg)  &> 0 \label{eq:3sc4}
		\end{align}
		
		\item Admissibility condition on the left interface:
		\begin{equation}
		\begin{split}
		&\mu_0(\ep)\,\Big((\vr_{M-}+\ep)\,s(\vr_{M-}+\ep,p_M-\delta) - \vr_-\,s(\vr_-,p_-)\Big)\\
		&\quad\leq (\vr_{M-}+\ep)\,s(\vr_{M-}+\ep,p_M-\delta)\,\mu_1(\ep) - \vr_-\,s(\vr_-,p_-)\,v_-
		\end{split}
		\label{eq:3adml}
		\end{equation}
		
		\item Admissibility condition on the right interface: 
		\begin{equation}
		\begin{split}
		&\mu_2(\ep)\,\Big(\vr_+\,s(\vr_+,p_+) - (\vr_{M+}-\ep)\,s(\vr_{M+}-\ep,p_M-\delta)\Big) \\
		&\quad\leq \vr_+\,s(\vr_+,p_+)\,v_+ - (\vr_{M+}-\ep)\,s(\vr_{M+}-\ep,p_M-\delta)\,\mu_1(\ep)
		\end{split}
		\label{eq:3admr}
		\end{equation}
	\end{itemize}
	\label{prop:3algequa}
\end{propo}

\begin{proof}
	Let there be $(\ep,\delta)\in (0,\ep_{\max}]\times(0,p_M)$ such that \eqref{eq:3order}-\eqref{eq:3admr} hold. In order to show that there exists an admissible $2$-fan subsolution, we use Proposition \ref{prop:algequa}, i.e. we define the constants appearing in Proposition \ref{prop:algequa} as follows:
	\begin{align*}
	\mu_0 &:= \mu_0(\ep); &
	\mu_1 &:= \mu_1(\ep); &
	\mu_2 &:= \mu_2(\ep); 
	\end{align*}	
	\begin{align*}
	\vr_1 &:= \vr_{M-} + \ep; &
	\vr_2 &:= \vr_{M+} - \ep; \\
	\vv_1 &:=\vv_2 := \left(\begin{array}{c} 0\\ \mu_1(\ep) \end{array}\right); & p_1 &:= p_2 := p_M - \delta; \\
	U_1 &:= \left(\begin{array}{rr} \gamma_1(\ep,\delta) & 0\\ 0 & -\gamma_1(\ep,\delta) \end{array}\right); &
	U_2 &:= \left(\begin{array}{rr} \gamma_2(\ep,\delta) & 0\\ 0 & -\gamma_2(\ep,\delta) \end{array}\right); \\
	C_1 &:= C_1(\ep,\delta); &
	C_2 &:= C_2(\ep,\delta).
	\end{align*}
	It is a matter of straightforward calculation to check that with this choice \eqref{eq:order}-\eqref{eq:admr} hold. 
\end{proof}

\subsubsection{Subsolution and admissibility inequalities}

In order to finish the proof of Proposition \ref{prop:mainWUE}, we have to find $(\ep,\delta)\in (0,\ep_{\max}]\times(0,p_M)$ such that the conditions \eqref{eq:3order}-\eqref{eq:3admr} are satisfied. 

We start with noting that we already have \eqref{eq:3order} fulfilled for all $\ep\in (0,\ep_{\max}]$.

Let us now investigate the subsolution conditions \eqref{eq:3sc1}-\eqref{eq:3sc4}. We start with the terms in the first parenthesis in \eqref{eq:3sc3}-\eqref{eq:3sc4}. We obtain by using that $\delta\in(0,p_M)$
\begin{align}
\frac{C_1(\ep,\delta)}{2} + \gamma_1(\ep,\delta) &= \frac{(2c_v-1)\,\delta}{\vr_{M-} + \ep} - \frac{2\mu_1 (\ep)\, \delta}{(\vr_{M-} + \ep)\big(\mu_0(\ep)-\mu_1(\ep)\big)} + \frac{C_1(\ep,0)}{2} + \gamma_1(\ep,0) \notag\\ 
&\geq \frac{(2c_v-1)\,\delta}{\vr_{M-} + \ep} \underbrace{- \frac{2|\mu_1 (\ep)|\, p_M}{(\vr_{M-} + \ep)\ \big|\mu_0(\ep)-\mu_1(\ep)\big|} + \frac{C_1(\ep,0)}{2} + \gamma_1(\ep,0)}_{=:R_1(\ep)} \\ 
\frac{C_2(\ep,\delta)}{2} + \gamma_2(\ep,\delta) &= \frac{(2c_v-1)\,\delta}{\vr_{M+} - \ep} - \frac{2\mu_1 (\ep)\, \delta}{(\vr_{M+} - \ep)\big(\mu_2(\ep)-\mu_1(\ep)\big)} + \frac{C_2(\ep,0)}{2} + \gamma_2(\ep,0) \notag\\ 
&\geq \frac{(2c_v-1)\,\delta}{\vr_{M+} - \ep} \underbrace{- \frac{2|\mu_1 (\ep)|\, p_M}{(\vr_{M+} - \ep)\ \big|\mu_2(\ep)-\mu_1(\ep)\big|} + \frac{C_2(\ep,0)}{2} + \gamma_2(\ep,0)}_{=:R_2(\ep)}
\end{align}	
where Propositions \ref{prop:limittostandard1} and \ref{prop:limittostandard2} together with the fact that $v_M=0$ imply that 
\begin{align*}
\lim\limits_{\ep\to 0 } R_1(\ep) &= 0; & \lim\limits_{\ep\to 0 } R_2(\ep) &= 0.
\end{align*}
Therefore $|R_1(\ep)|$ and $|R_2(\ep)|$ become arbitrary small if we choose $\ep$ small. Because of $c_v>\half$, there exists $\til{\ep}_1(\delta)\in(0,\ep_{\max}]$ for each $\delta\in(0,p_M)$, such that 
\begin{equation}\label{eq:3T1}
\frac{C_1(\ep,\delta)}{2} + \gamma_1(\ep,\delta) > 0 \qquad \& \qquad \frac{C_2(\ep,\delta)}{2} + \gamma_2(\ep,\delta) > 0
\end{equation}
hold for all $\ep\in(0,\til{\ep}_1(\delta))$. 

Similarly we handle terms in the second parenthesis in inequalities \eqref{eq:3sc3}-\eqref{eq:3sc4}. We obtain
\begin{align}
\frac{C_1(\ep,\delta)}{2} - \mu_1(\ep)^2 - \gamma_1(\ep,\delta) &= \frac{\delta}{\vr_{M-}+\ep} + \underbrace{\frac{C_1(\ep,0)}{2} - \mu_1(\ep)^2 - \gamma_1(\ep,0)}_{=:R_3(\ep)} \\
\frac{C_2(\ep,\delta)}{2} - \mu_1(\ep)^2 - \gamma_2(\ep,\delta) &= \frac{\delta}{\vr_{M+} - \ep} + \underbrace{\frac{C_2(\ep,0)}{2} - \mu_1(\ep)^2 - \gamma_2(\ep,0)}_{=:R_4(\ep)}.
\end{align}
With the same arguments as above, we obtain that for each $\delta \in(0,p_M)$ there exists $\til{\ep}_2(\delta)\in (0,\ep_{\max}]$ such that
\begin{equation}\label{eq:3T2}
\frac{C_i(\ep,\delta)}{2} - \mu_1(\ep)^2 - \gamma_i(\ep,\delta) > 0, \qquad i = 1,2
\end{equation}
hold for all $\ep\in(0,\til{\ep}_2(\delta))$. 

Combining \eqref{eq:3T1} and \eqref{eq:3T2} we obtain \eqref{eq:3sc3} and \eqref{eq:3sc4} while summing together \eqref{eq:3T1} and \eqref{eq:3T2} we obtain \eqref{eq:3sc1} and \eqref{eq:3sc2}.

To finish the proof we have to show that we can achieve that in addition the admissibility conditions \eqref{eq:3adml} and \eqref{eq:3admr} hold. Note that in the limit $(\ep,\delta)\to (0,0)$ the admissibility conditions \eqref{eq:3adml} and \eqref{eq:3admr} turn into the admissibility conditions of the 1D Riemann solution (according to Proposition \ref{prop:limittostandard1}). Since the latter are fulfilled strictly, we can choose $\delta>0$ sufficiently small and also $\ep\in(0,\min\{\til{\ep_1}(\delta),\til{\ep_2}(\delta)\})$ sufficiently small such that \eqref{eq:3adml} and \eqref{eq:3admr} hold. This finishes the proof of Proposition \ref{prop:mainWUE}.

\section{Proof of Theorem \ref{t:main}}\label{s:proofA}

Unlike in the case of $2$-fan solutions as in Theorem \ref{thm:mainWUE}, we cannot use here the 1D Riemann solution and try to perturb it in a suitable way in order to find a $1$-fan subsolution. On the other hand, since we don't have the middle interface, the number of equations and inequalities in Proposition \ref{prop:algequa} is lower than in the case $N = 2$.

Since $i$ could be only equal to $1$ in Proposition \ref{prop:algequa}, for simplicity of notation we write $\alpha$, $\beta$, $\gamma$, $\delta$ and $C$ instead of $\alpha_1$, $\beta_1$, $\gamma_1$, $\delta_1$ and $C_1$. We also recall that we set $\alpha = 0$ and $\delta = \alpha\beta = 0$, see \eqref{eq:choice}.

The subsolution conditions \eqref{eq:sc1}-\eqref{eq:sc2} then simplify to 
\begin{align}
C - \beta^2 &> 0 \label{eq:scs1};\\
\bigg(\frac{C}{2} + \gamma\bigg) \bigg(\frac{C}{2} - \beta^2 - \gamma\bigg)  &> 0. \label{eq:scs2} 
\end{align}
It is not difficult to observe (see also \cite[Lemma 4.3]{ChiKre14}) that the necessary condition for \eqref{eq:scs1}-\eqref{eq:scs2} to be satisfied is $\frac{C}{2} - \gamma> \beta^2$ which motivates us to introduce instead of $C$ and $\gamma$ new unknowns
\begin{align}
\ep_1 &:= \frac{C}{2} - \gamma - \beta^2; \\
\ep_2 &:= C - \beta^2 - \ep_1.
\end{align}

The set of algebraic identities and inequalities from Proposition \ref{prop:algequa} then simplifies into
\begin{itemize}
	\item Order of the speeds
	\begin{equation}
	\mu_0 < \mu_1
	\end{equation}
	\item Rankine Hugoniot conditions on the left interface:
	\begin{align}
	&\mu_0 (\vr_- - \vr_1) \, =\,  \vr_- v_- -\vr_1  \beta \label{eq:cont_left2}  \\
	&\mu_0 (\vr_- v_- - \vr_1 \beta) \, = \,  
	\vr_- v_-^2 - \vr_1(\beta^2 + \ep_1) + p_- - p_1 \label{eq:mom_2_left2} \\
	& \mu_0\left(2c_v p_- + \vr_- v_-^2 - 2c_v p_1 - \vr_1(\beta^2+\ep_1+\ep_2)\right)\nonumber\\
	&\qquad =  \big((2c_v+2) p_- + \vr_- v_-^2\big) v_- - \big((2c_v+2) p_1 + \vr_1(\beta^2+\ep_1+\ep_2)\big) \beta\, ; \label{eq:E_left2}
	\end{align}
	\item Rankine-Hugoniot conditions on the right interface:
	\begin{align}
	&\mu_1 (\vr_1-\vr_+ ) \, =\,  \vr_1  \beta - \vr_+ v_+ \label{eq:cont_right2}\\
	&\mu_1 (\vr_1 \beta- \vr_+ v_+) \, = \, \vr_1(\beta^2+\ep_1) - \vr_+ v_+^2 + p_1 - p_+ \label{eq:mom_2_right2} \\
	&\mu_1\left( 2c_vp_1 + \vr_1(\beta^2+\ep_1+\ep_2) - 2c_v p_+ - \vr_+ v_+^2 \right)\nonumber\\
	& \qquad =  \big((2c_v+2) p_1 + \vr_1(\beta^2+\ep_1+\ep_2)\big) \beta - \big((2c_v+2) p_+ + \vr_+ v_+^2\big) v_+\, ; \label{eq:E_right2}
	\end{align}
	\item Subsolution conditions:
	\begin{align}
	&\ep_1 > 0 \label{eq:sub_trace2}\\
	&\ep_2 > 0\, ;\label{eq:sub_det2}
	\end{align}
	\item Admissibility conditions
	\begin{align}
	&\mu_0(L_- - L_1) \geq L_-v_{-} - L_1\beta\label{eq:adm_left2} \\ 
	&\mu_1(L_1 - L_+) \geq L_1\beta - L_+v_{+}\, . \label{eq:adm_right2}
	\end{align}
\end{itemize}
where for simplicity of notation we introduced $L_i$ as
$$
L_i = \vr_is(\vr_i,p_i) = \vr_i \log \left(\frac{p_i^{c_v}}{\vr_i^{c_v+1}}\right), \qquad i=-,1,+.
$$
We observe that we have 6 equations and 5 inequalities for 7 unknowns $\mu_0,\mu_1,\vr_1,p_1,\beta,\ep_1,\ep_2$. Moreover to these 5 inequalities we have to keep in mind other inequalities which have to be satisfied, namely $\vr_1 > 0$ and $p_1 > 0$.

We define the following quantities, functions of the Riemann initial data
\begin{align}
R &= \vr_- - \vr_+ \\
A &= \vr_-v_- - \vr_+v_+ \\
H &= \vr_-v_-^2 - \vr_+v_+^2 + p_- - p_+ \\
u &= v_- - v_+ \\
B &= A^2 - RH = \vr_-\vr_+ u^2 - (\vr_+ - \vr_-)(p_+ - p_-).
\end{align}

Since we have six equations for seven unknowns it is reasonable to choose one unknown as a parameter and express other unknowns as functions of Riemann initial data and this parameter. We choose the density in the middle region $\vr_1$ as a parameter. Note that we may assume $B > 0$, this can be achieved by taking $u$ large enough as in the assumptions of Theorem \ref{t:main}.

\subsection{The case $R < 0$}

\subsubsection{Solution to algebraic equations}

Recalling that $R < 0$ means $\vr_- < \vr_+$ we express $\mu_0$ and $\mu_1$ from equations \eqref{eq:cont_left2}, \eqref{eq:mom_2_left2}, \eqref{eq:cont_right2}, \eqref{eq:mom_2_right2} as 
\begin{align}
\mu_0 = \frac{A}{R} - \frac{1}{R}\sqrt{B\frac{\vr_1-\vr_+}{\vr_1-\vr_-}} \label{eq:num}\\
\mu_1 = \frac{A}{R} - \frac{1}{R}\sqrt{B\frac{\vr_1-\vr_-}{\vr_1-\vr_+}} \label{eq:nup}
\end{align}
and we observe that $\mu_0 < \mu_1$ if $\vr_1 > \vr_+ > \vr_-$. In what follows we therefore always assume that the set of possible values of $\vr_1$ is $\vr_1 \in (\vr_+,\infty)$. There exists also another solution $\mu_0,\mu_1$, which we do not use, see Remark \ref{r:mu} at the end of section \ref{sss:AI}.

Next we express $\beta$ in two ways which will be useful later. Using \eqref{eq:cont_right2} and \eqref{eq:cont_left2} respectively we obtain
\begin{align}
\beta &= \frac{\vr_-}{\vr_1}v_- + \mu_0\frac{\vr_1-\vr_-}{\vr_1} \label{eq:beta1};\\
\beta &= \frac{\vr_+}{\vr_1}v_+ + \mu_1\frac{\vr_1-\vr_+}{\vr_1} \label{eq:beta2}.
\end{align}
It is technically more difficult to express $p_1,\ep_1$ and $\ep_2$. We start with rewriting \eqref{eq:mom_2_right2} as
\begin{equation}\label{eq:A}
p_1 = p_+ + \vr_+v_+^2 - \vr_1(\beta^2+\ep_1) + \mu_1^2(\vr_1-\vr_+).
\end{equation}
Then, a lengthy yet straightforward computation yields the following equations as a consequence of \eqref{eq:E_left2}, \eqref{eq:E_right2}
\begin{align}
&(\beta+v_-)\vr_1\ep_1(\vr_--\vr_1) - \vr_-\vr_1(\beta - v_-)(\ep_1+\ep_2) \nonumber \\
& \qquad = (\beta - v_-)\left[(\vr_--\vr_1)(p_1+p_-) + 2c_v\vr_-p_1 - 2c_v\vr_1p_- \right]; \label{eq:B} \\
-&(v_+ + \beta)\vr_1\ep_1(\vr_1-\vr_+) + \vr_+\vr_1(v_+ - \beta)(\ep_1+\ep_2) \nonumber \\
& \qquad = (v_+-\beta)\left[(\vr_1-\vr_+)(p_1+p_+) + 2c_v\vr_1p_+ - 2c_v\vr_+p_1 \right]. \label{eq:C}
\end{align}

At this point we assume that $v_+ - \beta$ and $v_- - \beta$ are nonzero and refer to Lemma \ref{l:61} for proof. We continue by expressing $\ep_1 + \ep_2$ from \eqref{eq:C} and plugging this to \eqref{eq:B}. We have
\begin{equation}\label{eq:ep1ep2}
\ep_1 + \ep_2 = \frac{1}{\vr_1\vr_+}\left[(\vr_1-\vr_+)(p_1+p_+) + 2c_v\vr_1p_+ - 2c_v\vr_+p_1 \right] + \frac{v_++\beta}{v_+-\beta}\ep_1\frac{\vr_1-\vr_+}{\vr_+}
\end{equation}
and consequently \eqref{eq:B} becomes
\begin{align}
&\left(\frac{\beta+v_-}{\beta-v_-}\vr_1\vr_+(\vr_--\vr_1) - \frac{v_++\beta}{v_+-\beta}\vr_1\vr_-(\vr_1-\vr_+)\right)\ep_1 \nonumber \\
= &\vr_+\left[(\vr_--\vr_1)(p_1+p_-) + 2c_v(\vr_-p_1 - \vr_1p_-) \right] + \vr_-\left[(\vr_1-\vr_+)(p_1+p_+) + 2c_v(\vr_1p_+ - \vr_+p_1) \right]. \label{eq:B2}
\end{align}
Expressing the right hand side of \eqref{eq:B2} we end up with 
\begin{align}
&\left(\frac{\beta+v_-}{\beta-v_-}\vr_1\vr_+(\vr_--\vr_1) - \frac{v_++\beta}{v_+-\beta}\vr_1\vr_-(\vr_1-\vr_+)\right)\ep_1 \nonumber \\
= &\vr_1p_1R + \vr_+\vr_-(p_--p_+) + (2c_v+1)\vr_1(\vr_-p_+-\vr_+p_-). \label{eq:B3}
\end{align}
Once again it is useful to introduce further notation to simplify resulting expressions. We define
\begin{align}
X &= \vr_+\vr_-(p_--p_+) + (2c_v+1)\vr_1(\vr_-p_+-\vr_+p_-) \label{eq:X} \\
Y &= \frac{\beta+v_-}{\beta-v_-}\vr_+(\vr_--\vr_1) - \frac{v_++\beta}{v_+-\beta}\vr_-(\vr_1-\vr_+) \label{eq:Y} \\
Z &= p_+ + \vr_+v_+^2 - \vr_1\beta^2 + \mu_1^2(\vr_1-\vr_+). \label{eq:Z}
\end{align}
This way, assuming $Y \neq 0$ which will be justified in Lemma \ref{l:62}, we rewrite \eqref{eq:B3} as 
\begin{equation}
\ep_1 = \frac{\vr_1p_1R + X}{\vr_1 Y}
\end{equation}
and plugging this into \eqref{eq:A} and assuming $Y + \vr_1R \neq 0$ which will be justified in Lemma \ref{l:62}, too, we get
\begin{equation}
p_1 = \frac{YZ - X}{Y+\vr_1R}, \label{eq:p1}
\end{equation}
this in turn yields
\begin{equation} \label{eq:ep1}
\ep_1 = \frac{\vr_1RZ + X}{\vr_1(Y+\vr_1R)}
\end{equation}
and finally $\ep_2$ is expressed using \eqref{eq:ep1ep2}
\begin{equation} \label{eq:ep2}
\ep_2 = \frac{1}{\vr_1\vr_+}\left[(\vr_1-\vr_+)(p_1+p_+) + 2c_v(\vr_1p_+ - \vr_+p_1) \right] + \left(\frac{v_++\beta}{v_+-\beta}\frac{\vr_1-\vr_+}{\vr_+}-1\right)\ep_1.
\end{equation}

\subsubsection{Positivity of $p_1$, $\ep_1$ and $\ep_2$}\label{ss:ss}

Before we continue further, we remind the following useful expressions which can be derived from \eqref{eq:cont_left2} and \eqref{eq:cont_right2}
\begin{align}
\beta - \mu_0 &= \frac{\vr_-}{\vr_1}\left(v_- -\mu_0\right) \label{eq:ep2 1}\\
v_- - \beta &= \frac{\vr_1-\vr_-}{\vr_1}\left(v_--\mu_0\right) \label{eq:ep2 2}\\
\mu_1 - \beta &= \frac{\vr_+}{\vr_1}\left(\mu_1 - v_+\right) \label{eq:ep2 3}\\
\beta - v_+ &= \frac{\vr_1-\vr_+}{\vr_1}\left(\mu_1-v_+\right). \label{eq:ep2 4}
\end{align}
In particular we observe that the signs of $v_- - \beta$ and $\beta - v_+$ are the same as $v_- - \mu_0$ and $\mu_1 - v_+$ respectively. 
\begin{lemma}\label{l:61}
	For $u = v_- - v_+$ sufficiently large it holds 
	\begin{equation}
	v_- - \mu_0 > 0 \qquad \& \qquad \mu_1 - v_+ > 0.\label{eq:642}
	\end{equation}
\end{lemma}
\begin{proof}
	We have
	\begin{align}
	v_- - \mu_0 &= -\frac{\vr_+ u}{R} + \frac{1}{R}\sqrt{\big(\vr_-\vr_+ u^2 - (\vr_+-\vr_-)(p_+-p_-)\big)\frac{\vr_1-\vr_+}{\vr_1-\vr_-}} \nonumber\\
	&= \frac{\vr_+ u}{\abs{R}} - \frac{1}{\abs{R}}\sqrt{\big(\vr_-\vr_+ u^2 - (\vr_+-\vr_-)(p_+-p_-)\big)\frac{\vr_1-\vr_+}{\vr_1-\vr_-}}.
	\end{align}
	Now we distinguish two cases. First, if $p_+ \geq p_-$, then we easily observe (recall $\vr_+ > \vr_-$) that
	\begin{equation}
	v_- - \mu_0 > \frac{1}{\abs{R}}(\vr_+ u - \sqrt{\vr_-\vr_+} u) > 0.
	\end{equation}
	However if $p_+ < p_-$ we get
	\begin{equation}
	v_- - \mu_0 > \frac{1}{\abs{R}}(\vr_+u - \sqrt{\vr_-\vr_+u^2 + (\vr_+-\vr_-)(p_--p_+) })
	\end{equation}
	and the expression on the right hand side can be made positive assuming $u$ is large enough. Similarly we proceed with quantity $\mu_+ - v_+$.
\end{proof}
We recall that Lemma \ref{l:61} justifies the assumptions $v_+ - \beta$ and $v_- - \beta$ nonzero we made in the process of deriving \eqref{eq:p1} and \eqref{eq:ep1}.

We continue by analyzing whether one can find $\vr_1$ such that $p_1$, $\ep_1$ and $\ep_2$ are all positive. This part starts with rewriting the expression \eqref{eq:Y} for $Y$. We have
\begin{align}
Y &= \vr_+(\vr_1-\vr_-)\frac{v_-+\beta}{v_--\beta} + \vr_-(\vr_1-\vr_+)\frac{\beta+v_+}{\beta-v_+} \nonumber \\
&= \vr_1\vr_+ \frac{v_-+\beta}{v_--\mu_0} + \vr_1\vr_- \frac{\beta+v_+}{\mu_1-v_+} \nonumber \\
&= \vr_1\vr_+ \frac{v_-+\mu_0}{v_--\mu_0} + \vr_1\vr_- \frac{\mu_1 + v_+}{\mu_1-v_+}, \label{eq:646}
\end{align}
where the last equality holds although obviously $\beta \notin \{\mu_0,\mu_1\}$. In what follows our strategy is to express certain quantities as functions of $\vr_1$ and $u$ while treating $\vr_\pm,p_\pm$ and $v_-$ as data. This way we obtain
\begin{equation}\label{eq:YY}
Y = \vr_1R\left[1+\frac{2}{\sqrt{\frac{B}{\vr_-^2 u^2}\frac{\vr_1-\vr_-}{\vr_1-\vr_+}}-1} + 2v_-\left(\frac{1}{\sqrt{\frac{B(\vr_1-\vr_+)}{\vr_+^2(\vr_1-\vr_-)}}-u} - \frac{1}{\sqrt{\frac{B(\vr_1-\vr_-)}{\vr_-^2(\vr_1-\vr_+)}}-u}\right)\right]. 
\end{equation}

\begin{lemma}\label{l:62}
	For $u = v_- - v_+$ large enough it holds $Y < 0$ and $Y + \vr_1R < 0$ for all $\vr_1 > \vr_+$.
\end{lemma}
\begin{proof}
	To prove Lemma \ref{l:62} we examine the limit of $Y$ as $u \rightarrow \infty$. First we observe
	\begin{equation}\label{eq:YY2}
	\sqrt{\frac{B}{\vr_-^2 u^2}\frac{\vr_1-\vr_-}{\vr_1-\vr_+}} \rightarrow \sqrt{\frac{\vr_+(\vr_1-\vr_-)}{\vr_-(\vr_1-\vr_+)}} > 1
	\end{equation}
	and therefore the second term in the square brackets of \eqref{eq:YY} is positive. We continue by proving that the last term converges to zero. We write
	\begin{align}
	&2v_-\left(\frac{1}{\sqrt{\frac{B(\vr_1-\vr_+)}{\vr_+^2(\vr_1-\vr_-)}}-u} - \frac{1}{\sqrt{\frac{B(\vr_1-\vr_-)}{\vr_-^2(\vr_1-\vr_+)}}-u}\right) \nonumber \\
	&\qquad  = \frac{2v_-}{u}\left(\frac{1}{\sqrt{\frac{B(\vr_1-\vr_+)}{u^2\vr_+^2(\vr_1-\vr_-)}}-1} - \frac{1}{\sqrt{\frac{B(\vr_1-\vr_-)}{\vr_-^2u^2(\vr_1-\vr_+)}}-1}\right) \nonumber \\
	&\qquad  = \frac{2v_-}{u}\left(\frac{\sqrt{\frac{B(\vr_1-\vr_-)}{\vr_-^2u^2(\vr_1-\vr_+)}} - \sqrt{\frac{B(\vr_1-\vr_+)}{u^2\vr_+^2(\vr_1-\vr_-)}}}{\left(\sqrt{\frac{B(\vr_1-\vr_+)}{u^2\vr_+^2(\vr_1-\vr_-)}}-1\right)\left(\sqrt{\frac{B(\vr_1-\vr_-)}{\vr_-^2u^2(\vr_1-\vr_+)}}-1\right)}\right)\label{eq:YY3}
	\end{align}
	and observe that both terms in the numerator of the fraction converge to distinct quantities and the denominator has finite limit as $u \rightarrow \infty$. Therefore the whole fraction has a finite nonzero limit and the whole term therefore can be made arbitrarily small by choosing $u$ sufficiently large.
	
	Altogether we have proved that the expression in the square brackets in \eqref{eq:YY} is positive at least for $u$ sufficiently large and since $R < 0$ we conclude that $Y < 0$ for large $u$. Also we observe that $Y$ has a finite limit as $u \rightarrow \infty$. The same two properties obviously hold also for $Y + \vr_1R$.
\end{proof}

\begin{remark}
	At this point we emphasize the need to take $v_-$ as fixed and by taking $u$ large enough recover $v_+$. This procedure would not work the other way round, i.e. taking $v_+$ fixed. If we would consider $v_+$ fixed, the expression \eqref{eq:YY} would become
	\begin{equation}\label{eq:YY4}
	Y = \vr_1R\left[1+\frac{2}{\sqrt{\frac{B}{\vr_+^2 u^2}\frac{\vr_1-\vr_+}{\vr_1-\vr_-}}-1} + 2v_+\left(\frac{1}{\sqrt{\frac{B(\vr_1-\vr_+)}{\vr_+^2(\vr_1-\vr_-)}}-u} - \frac{1}{\sqrt{\frac{B(\vr_1-\vr_-)}{\vr_-^2(\vr_1-\vr_+)}}-u}\right)\right]. 
	\end{equation}
	It is not difficult to figure out that 
	\begin{equation}
	\lim_{u \rightarrow \infty} \left(1+\frac{2}{\sqrt{\frac{B}{\vr_+^2 u^2}\frac{\vr_1-\vr_+}{\vr_1-\vr_-}}-1}\right) < 0,
	\end{equation}
	which yields $Y$ positive and this would cause troubles later when studying the sign of $\ep_1$.
\end{remark}

Now let us turn our attention to expressions $X$ and $Z$. We recall here the definition \eqref{eq:X} of $X$ 
\begin{equation}
X = \vr_+\vr_-(p_--p_+) + (2c_v+1)\vr_1(\vr_-p_+-\vr_+p_-).
\end{equation}
We see that $X$ does not depend on $u$ and may be positive or negative depending on values of $\vr_\pm, p_\pm$ and $\vr_1$.

Concerning $Z$ we have from \eqref{eq:Z}
\begin{align}
Z &= p_+ + \vr_+v_+^2 - \vr_1\beta^2 + \mu_1^2(\vr_1-\vr_+) \nonumber \\
&= p_+ + \frac{(\vr_1-\vr_+)\vr_+}{\vr_1}(v_+-\mu_1)^2
\end{align}
and we see that $Z$ is always positive. Moreover we plug in the expression \eqref{eq:nup} to obtain
\begin{align}
Z &= p_+ + \frac{(\vr_1-\vr_+)\vr_+}{\vr_1}(v_+-\mu_1)^2 \nonumber \\
&= p_+ + \frac{(\vr_1-\vr_+)\vr_+}{\vr_1R^2}\left(\sqrt{B\frac{\vr_1-\vr_-}{\vr_1-\vr_+}} - \vr_- u\right)^2 \nonumber \\
&= p_+ + \frac{(\vr_1-\vr_+)\vr_+u^2}{\vr_1R^2}\left(\sqrt{\left(\vr_-\vr_+ +\frac{R(p_+-p_-)}{u^2}\right)\frac{\vr_1-\vr_-}{\vr_1-\vr_+}} - \vr_-\right)^2.\label{eq:blabla}
\end{align}
Observing that the quantity in parenthesis on the last line of \eqref{eq:blabla} has finite nonzero limit as $u \rightarrow \infty$ we conclude that $Z$ grows as $u^2$.

We are now ready to study the signs of $p_1$ and $\ep_1$. 
\begin{lemma}\label{l:63}
	For $u = v_- - v_+$ large enough it holds $p_1 > 0$ and $\ep_1 > 0$ for all $\vr_1 > \vr_+$.
\end{lemma}
\begin{proof}
	Since we have from \eqref{eq:p1}
	\begin{equation}
	p_1 = \frac{YZ - X}{Y+\vr_1R}, \label{eq:p12a}
	\end{equation}
	we easily see that for large $u$ the leading term is $\frac{YZ}{Y+\vr_1R}$ which grows like $u^2$, whereas $\frac{X}{Y+\vr_1R}$ has a finite limit as $u \rightarrow \infty$. Since both $Y$ and $Y+\vr_1R$ are negative at least for large $u$ and $Z > 0$, we conclude that $p_1 > 0$ for $u$ large enough.
	
	Similarly we have from \eqref{eq:ep1}
	\begin{equation} \label{eq:ep12a}
	\ep_1 = \frac{\vr_1RZ + X}{\vr_1(Y+\vr_1R)}
	\end{equation}
	and using similar arguments as above for $p_1$ we see that the term $\frac{RZ}{Y+\vr_1R}$ grows like $u^2$ and is positive whereas $\frac{X}{\vr_1(Y+\vr_1R)}$ has finite limit as $u \rightarrow \infty$ and therefore $\ep_1 > 0$ for $u$ large enough.
\end{proof}

Next we study the sign of $\ep_2$. We have
\begin{lemma}\label{l:64}
	For  
	\begin{equation}\label{eq:rho1_condA}
	\rho_1 > \frac{(2c_v+1)\vr_-+\vr_+}{2} + \frac{\sqrt{\vr_+^2 + (4c_v^2-1)\vr_+\vr_-}}{2}
	\end{equation}
	and $u = v_- - v_+$ large enough it holds $\ep_2 > 0$.
\end{lemma}
\begin{proof}
	We start from \eqref{eq:ep2}
	\begin{align}
	\ep_2 &= \frac{1}{\vr_1\vr_+}\left[(\vr_1-\vr_+)(p_1+p_+) + 2c_v(\vr_1p_+ - \vr_+p_1) \right] + \ep_1\left(\frac{v_++\beta}{v_+-\beta}\frac{\vr_1-\vr_+}{\vr_+}-1\right) \nonumber \\
	&= \frac{p_+((2c_v+1)\vr_1-\vr_+)}{\vr_1\vr_+} + \frac{p_1(\vr_1-(2c_v+1)\vr_+)}{\vr_1\vr_+} \nonumber \\
	&\qquad + \ep_1\left(-\frac{\vr_1}{\vr_+} + \frac{2\vr_1R}{\vr_+\left(\vr_- - \sqrt{\frac{B(\vr_1-\vr_-)}{u^2(\vr_1-\vr_+)}}\right)} + \frac{2\vr_1Rv_-}{\vr_+\left(\sqrt{B\frac{\vr_1-\vr_-}{\vr_1-\vr_+}} - \vr_- u\right)}\right).\label{eq:73}
	\end{align}
	
	Keeping in mind that both $p_1$ and $\ep_1$ grow as $u^2$ we identify the leading terms of $\ep_2$ as $u \rightarrow \infty$. We have 
	\begin{align}
	\ep_2 &= \frac{YZ}{Y+\vr_1R}\frac{\vr_1-(2c_v+1)\vr_+}{\vr_1\vr_+} + \frac{RZ}{Y+\vr_1R}\left(-\frac{\vr_1}{\vr_+} + \frac{2\vr_1R}{\vr_+\left(\vr_- - \sqrt{\frac{B(\vr_1-\vr_-)}{u^2(\vr_1-\vr_+)}}\right)}\right) + l.o.t. \nonumber \\
	&= \frac{Z}{\vr_+(Y+\vr_1R)}\left(\frac{Y(\vr_1-(2c_v+1)\vr_+)}{\vr_1} + R\left(-\vr_1 + \frac{2\vr_1R}{\vr_+\left(\vr_- - \sqrt{\frac{B(\vr_1-\vr_-)}{u^2(\vr_1-\vr_+)}}\right)}\right)\right) + l.o.t. \nonumber \\
	&=:\frac{Z}{\vr_+(Y+\vr_1R)} E + l.o.t., \label{eq:74}
	\end{align}
	where $l.o.t.$ stands for lower order terms with respect to $u^2$. We know that $\frac{Z}{\vr_+(Y+\vr_1R)}$ is negative so we continue by studying the remaining part $E$ of the leading term of $\ep_2$ and to make $\ep_2 > 0$ we want to show that $E < 0$.
	
	Recall that from \eqref{eq:YY}-\eqref{eq:YY3} we already know that
	\begin{equation}
	\lim_{u \rightarrow \infty} Y = \vr_1R\left(1+\frac{2}{\sqrt{\frac{\vr_+(\vr_1-\vr_-)}{\vr_-(\vr_1-\vr_+)}}-1}\right)
	\end{equation}
	and therefore a straightforward calculation yields
	\begin{equation}\label{eq:limE}
	\lim_{u \rightarrow \infty} E = R\left[-(2c_v+1)\vr_+ + \left(\frac{2\vr_+\vr_1}{\vr_-} - 2(2c_v+1)\vr_+\right)\frac{1}{\sqrt{\frac{\vr_+(\vr_1-\vr_-)}{\vr_-(\vr_1-\vr_+)}}-1}\right]
	\end{equation}
	Since we have $R < 0$, we want the term in square parenthesis of \eqref{eq:limE} to be positive. It is not difficult to show that this can be done by choosing $\vr_1$ large enough with respect to $\vr_\pm$, more specifically if 
	\begin{equation}
	\vr_1 > \frac{(2c_v+1)\vr_-+\vr_+}{2} + \frac{\sqrt{\vr_+^2 + (4c_v^2-1)\vr_+\vr_-}}{2}
	\end{equation}
	then $\lim_{u \rightarrow \infty} E < 0$.
	
\end{proof}

\subsubsection{Admissibility inequalities} \label{sss:AI}

Finally we have to check whether the admissibility inequalities \eqref{eq:adm_left2}--\eqref{eq:adm_right2} are satisfied. We start by plugging the expression \eqref{eq:beta1} into \eqref{eq:adm_left2} and we obtain
\begin{equation}
(v_- - \mu_0)(\vr_-L_1 - \vr_1L_-) \geq 0.
\end{equation}
We already know that $v_- - \mu_0 > 0$ at least for $u$ large enough, see Lemma \ref{l:61}. Therefore we need to ensure that 
\begin{equation}\label{eq:AA1}
\vr_-L_1 \geq \vr_1L_-.
\end{equation}
Using the definition of $L_i$ the inequality \eqref{eq:AA1} is equivalent to
\begin{equation}\label{eq:AD1}
\frac{p_1^{c_v}}{\vr_1^{c_v+1}} \geq \frac{p_-^{c_v}}{\vr_-^{c_v+1}}.
\end{equation}

We proceed similarly with \eqref{eq:adm_right2} where we plug in the expression \eqref{eq:beta2} for $\beta$ to get
\begin{equation}
(\mu_1 - v_+)(\vr_+L_1 - \vr_1L_+) \geq 0.
\end{equation}
Again by Lemma \ref{l:61} we know that $\mu_1 - v_+ > 0$ at least for $u$ large enough and thus we need
\begin{equation}\label{eq:AA2}
\vr_+L_1 \geq \vr_1L_+,
\end{equation}
which is equivalent to 
\begin{equation}\label{eq:AD2}
\frac{p_1^{c_v}}{\vr_1^{c_v+1}} \geq \frac{p_+^{c_v}}{\vr_+^{c_v+1}}.
\end{equation}

Combining \eqref{eq:AD1} and \eqref{eq:AD2} we end up with
\begin{equation}\label{eq:AD12}
p_1^{c_v} \geq \vr_1^{c_v+1}\max \bigg\{\frac{p_-^{c_v}}{\vr_-^{c_v+1}},\frac{p_+^{c_v}}{\vr_+^{c_v+1}}\bigg\}.
\end{equation}
However we already know that $p_1$ grows as $u^2$, so regardless of the choice of $\vr_1$ we can always ensure that \eqref{eq:AD12} is satisfied by choosing $u$ large enough.

We have showed that for any $\vr_1$ satisfying \eqref{eq:rho1_condA} there exists some $\overline{u}(\vr_1)$ such that if $u > \overline{u}(\vr_1)$ then we can construct a $1$-fan subsolution. Then the point $(i)$ of Theorem \ref{t:main} in the case $R < 0$ is proved by defining
\begin{equation}
U = U(\vr_\pm,p_\pm,v_-) := v_- - \inf_{\rho_1}\overline{u}(\rho_1),
\end{equation}
where the $\inf$ is taken among $\rho_1$ such that \eqref{eq:rho1_condA} holds.

\begin{remark}\label{r:mu}
	There exists also another solution to equations \eqref{eq:cont_left2}, \eqref{eq:mom_2_left2}, \eqref{eq:cont_right2}, \eqref{eq:mom_2_right2}, namely 
	\begin{align}
	\mu_0 = \frac{A}{R} + \frac{1}{R}\sqrt{B\frac{\vr_1-\vr_+}{\vr_1-\vr_-}} \label{eq:num5}\\
	\mu_1 = \frac{A}{R} + \frac{1}{R}\sqrt{B\frac{\vr_1-\vr_-}{\vr_1-\vr_+}} \label{eq:nup5}
	\end{align}
	and here it holds $\mu_0 < \mu_1$ for $\vr_1 < \vr_- < \vr_+$. However this solution is not convenient to work with, since it violates Lemma \ref{l:61}. More precisely, one of the inequalities \eqref{eq:642} holds with the opposite sign. This causes troubles in the analysis of the admissibility inequalities \eqref{eq:adm_left2}--\eqref{eq:adm_right2}, because these then yield one upper bound and one lower bound for $p_1$ instead of two lower bounds. Then naturally the argument with taking $u$ sufficiently large fails.
\end{remark}

\subsection{The case $R > 0$}

The proof in the case $R > 0$ follows the same steps. The expressions \eqref{eq:num}--\eqref{eq:nup} for $\mu_0$, $\mu_1$ stay the same and again we see that $\mu_0 < \mu_1$ if $\vr_1 > \vr_- > \vr_+$. Since we treat $v_+$ as fixed in this case, instead of \eqref{eq:YY} we get
\begin{equation}\label{eq:YYY}
Y = \vr_1R\left[1+\frac{2}{\sqrt{\frac{B}{\vr_+^2 u^2}\frac{\vr_1-\vr_+}{\vr_1-\vr_-}}-1} + 2v_+\left(\frac{1}{\sqrt{\frac{B(\vr_1-\vr_+)}{\vr_+^2(\vr_1-\vr_-)}}-u} - \frac{1}{\sqrt{\frac{B(\vr_1-\vr_-)}{\vr_-^2(\vr_1-\vr_+)}}-u}\right)\right]. 
\end{equation}
Similarly as in the case $R < 0$ we show that the term in the square brackets is positive at least for $u$ sufficiently large, so in this case we conclude that $Y > 0$ and obviously also $Y + \vr_1R > 0$. Since nothing changes in terms $X$ and $Z$, this yields that $p_1$ and $\ep_1$ are both positive at least for $u$ sufficiently large.

The analysis of $\ep_2$ changes again with respect to the fact that $v_+$ is now fixed instead of $v_-$. Instead of \eqref{eq:73} we obtain
\begin{equation}
\ep_2 = \frac{p_+((2c_v+1)\vr_1-\vr_+)}{\vr_1\vr_+} + \frac{p_1(\vr_1-(2c_v+1)\vr_+)}{\vr_1\vr_+} + \ep_1\left(-\frac{\vr_1}{\vr_+} + \frac{2\vr_1Rv_+}{\vr_+\left(\sqrt{B\frac{\vr_1-\vr_-}{\vr_1-\vr_+}} - \vr_- u\right)}\right)
\end{equation}
and hence
\begin{equation}
\ep_2 = \frac{Z}{\vr_+(Y+\vr_1R)}\left(\frac{Y(\vr_1-(2c_v+1)\vr_+)}{\vr_1} -\vr_1R\right) + l.o.t. =: \frac{Z}{\vr_+(Y+\vr_1R)} \tilde{E} + l.o.t.. \label{eq:742}
\end{equation}
Since the limit of $Y$ as $u \rightarrow \infty$ changes to
\begin{equation}
\lim_{u \rightarrow \infty} Y = \vr_1R\left(1+\frac{2}{\sqrt{\frac{\vr_-(\vr_1-\vr_+)}{\vr_+(\vr_1-\vr_-)}}-1}\right),
\end{equation}
a direct calculation yields 
\begin{equation}\label{eq:limEtilda}
\lim_{u \rightarrow \infty} \tilde{E} = R\left[-(2c_v+1)\vr_+ + \frac{2\vr_1 - 2(2c_v+1)\vr_+}{\sqrt{\frac{\vr_-(\vr_1-\vr_+)}{\vr_+(\vr_1-\vr_-)}}-1}\right].
\end{equation}
It is not difficult to obtain that $\lim_{u \rightarrow \infty} \tilde{E} > 0$ if 
\begin{equation}\label{eq:rho1_cond2}
\vr_1 > \frac{\vr_-+(2c_v+1)\vr_+}{2} + \frac{\sqrt{\vr_-^2 + (4c_v^2-1)\vr_+\vr_-}}{2}
\end{equation}
and thus choosing such $\vr_1$ we always find $u$ large enough such that $\ep_2 > 0$. Since nothing changes regarding the admissibility inequalities, we conclude the proof of point $(ii)$ of Theorem \ref{t:main} in the case $R > 0$ by the same arguments as in the case of point $(i)$ in the case $R < 0$.

\subsection{The case $R = 0$}

\subsubsection{Solution to algebraic equations}

This case has to be treated separately since even the expression for the speeds of interfaces $\mu_0,\mu_1$ \eqref{eq:num}-\eqref{eq:nup} are considerably different. We have
\begin{align}
\mu_0 &= \frac{1}{2}\left( -\frac{\vr_+ u}{\vr_1-\vr_+} + \frac{p_--p_+}{\vr_+ u} + v_-+v_+\right) \label{eq:num2}\\
\mu_1 &= \frac{1}{2}\left(\frac{\vr_+ u}{\vr_1-\vr_+} + \frac{p_--p_+}{\vr_+ u} + v_-+v_+\right) \label{eq:nup2}
\end{align}
and we immediately see that since we examine the case where $u > 0$ we are forced to assume $\vr_1 > \vr_+ = \vr_-$ in order to ensure $\mu_0 < \mu_1$. 

The relations \eqref{eq:beta1}-\eqref{eq:B2} stay the same in the case $R = 0$, however the expressions for quantities $X$ and $Y$ simplify a little bit to 
\begin{align}
X &= \vr_+^2(p_--p_+) + (2c_v+1)\vr_1\vr_+(p_+-p_-) = \vr_+(p_+-p_-)\left((2c_v+1)\vr_1 - \vr_+\right) \label{eq:X2} \\
Y &= \vr_+(\vr_1-\vr_+)\left(\frac{v_- + \beta}{v_- - \beta} + \frac{\beta + v_+}{\beta - v_+}\right) \label{eq:Y2} 
\end{align}
and consequently we have
\begin{equation} \label{eq:ep12b}
\ep_1 = \frac{X}{\vr_1 Y}
\end{equation}
and 
\begin{equation}
p_1 = Z - \frac{X}{Y}, \label{eq:p12b}
\end{equation}
with $Z$ defined as in \eqref{eq:Z}. Finally, $\ep_2$ is given by \eqref{eq:ep2}.

\subsubsection{Positivity of $p_1, \ep_1$ and $\ep_2$}

We claim that Lemma \ref{l:61} holds also in the case $R = 0$. Indeed, it is even easier here to see that
\begin{align}
v_- - \mu_0 &= \frac{1}{2}\left(\frac{\vr_1 u}{\vr_1 - \vr_+} - \frac{p_--p_+}{\vr_+ u}\right) \\
\mu_1 - v_+ &= \frac{1}{2}\left(\frac{\vr_1 u}{\vr_1 - \vr_+} + \frac{p_--p_+}{\vr_+ u}\right)
\end{align}
and both expressions on the right hand sides are positive at least for $u$ large enough.

Next, we observe that the sign of $X$ depends on the sign of $p_+ - p_-$, namely 
\begin{equation}
\text{sign}\, X = \text{sign}\, (p_+ - p_-)
\end{equation}
and in particular this shows that this presented construction simply does not work in the case $\vr_- = \vr_+$ and $p_- = p_+$, because in this case we end up with $\ep_1 = 0$. Also, knowing the sign of $X$ and having in mind the expression \eqref{eq:ep12b} for $\ep_1$ we see that in order to ensure $\ep_1 > 0$ we need to have
\begin{equation}
\text{sign}\, Y = \text{sign}\, X = \text{sign}\, (p_+ - p_-).
\end{equation}

Therefore we study now the sign of $Y$ in a similar manner as in the proof of Lemma \ref{l:62}. We start this part by claiming that the expression \eqref{eq:646} still holds and can be written as 
\begin{equation}\label{eq:6462}
Y = \vr_1\vr_+\left(\frac{v_-+\mu_0}{v_--\mu_0} + \frac{\mu_1 + v_+}{\mu_1-v_+}\right).
\end{equation}
Using \eqref{eq:num2}, \eqref{eq:nup2}, keeping $v_-$ fixed and expressing $v_+ = v_- - u$ we end up after a simple calculation with
\begin{equation}\label{eq:688}
Y = 4\vr_1\vr_+\left(\frac{2v_-}{\frac{\vr_1 u}{\vr_1-\vr_+} - \frac{(\vr_1-\vr_+)(p_--p_+)^2}{\vr_1\vr_+^2 u^3}} - \frac{1}{\frac{\vr_1}{\vr_1-\vr_+} + \frac{p_--p_+}{\vr_+ u^2}}\right).
\end{equation}
We observe that in the limit $u \rightarrow \infty$ we get $Y \rightarrow -4\vr_+(\vr_1-\vr_+) < 0$. Therefore we conclude that $\ep_1 > 0$ for $u$ large enough if we keep $v_-$ fixed and if $p_- > p_+$.

Similarly if we keep $v_+$ fixed, express $v_- = v_+ + u$ we end up instead of \eqref{eq:688} with
\begin{equation}\label{eq:689}
Y = 4\vr_1\vr_+\left(\frac{2v_+}{\frac{\vr_1 u}{\vr_1-\vr_+} - \frac{(\vr_1-\vr_+)(p_--p_+)^2}{\vr_1\vr_+^2 u^3}} + \frac{1}{\frac{\vr_1}{\vr_1-\vr_+} - \frac{p_--p_+}{\vr_+ u^2}}\right)
\end{equation}
and in the limit $u \rightarrow \infty$ we get $Y \rightarrow 4\vr_+(\vr_1-\vr_+) > 0$. Therefore we conclude that $\ep_1 > 0$ for $u$ large enough if we keep $v_+$ fixed and if $p_- < p_+$.

Concerning the sign of $p_1$ we have similarly as in \eqref{eq:blabla}
\begin{equation}
Z = \frac{\vr_1\vr_+ u^2}{4(\vr_1-\vr_+)} + \frac{p_-+p_+}{2} + \frac{(\vr_1-\vr_+)(p_--p_+)^2}{4\vr_1\vr_+ u^2},
\end{equation}
thus we see that $Z > 0$ and $Z \sim u^2$ as $u \rightarrow \infty$. Therefore using \eqref{eq:p12b} we conclude $p_1 > 0$.

\begin{remark}\label{r:eqpress}
	From the considerations above it is also clear why we do not state anything about the case $\vr_- = \vr_+$ and $p_- = p_+$ in Theorem \ref{t:main}. In this case $X = 0$ and the only way how one could choose $\ep_1 > 0$ and satisfy \eqref{eq:B3} would be to have also $Y = 0$. However, even if $p_- = p_+$ we observe that $Y$ has nonzero finite limit as $u \rightarrow \infty$ and hence this construction fails in this case.
\end{remark}

We finish this section with the study of the sign of $\ep_2$. 
\begin{lemma}\label{l:642}
	For  
	\begin{equation}\label{eq:rho1_condB}
	\rho_1 > (2c_v+1)\vr_+
	\end{equation}
	and $u = v_- - v_+$ large enough it holds $\ep_2 > 0$.
\end{lemma}
\begin{proof}
	Recalling \eqref{eq:ep2} we have
	\begin{align} 
	\ep_2 &= \frac{1}{\vr_1\vr_+}\left[(\vr_1-\vr_+)(p_1+p_+) + 2c_v(\vr_1p_+ - \vr_+p_1) \right] + \ep_1\left(\frac{v_++\beta}{v_+-\beta}\frac{\vr_1-\vr_+}{\vr_+}-1\right) \nonumber \\
	&= p_1 \frac{\vr_1 - (1+2c_v)\vr_+}{\vr_1\vr_+} + p_+\frac{(2c_v+1)\vr_1-\vr_+}{\vr_1\vr_+} + \ep_1\left(\frac{v_++\beta}{v_+-\beta}\frac{\vr_1-\vr_+}{\vr_+}-1\right). \label{eq:ep2 11}
	\end{align}
	We already know, that $p_1$ grows as $u^2$ as $u \rightarrow \infty$. The second term on the right hand side of \eqref{eq:ep2 11} does not depend on $u$ at all. Concerning the third term, we know that $\ep_1$ has a finite limit as $u \rightarrow \infty$ and now we are going to show that the same property holds also for the term $\left(\frac{v_++\beta}{v_+-\beta}\frac{\vr_1-\vr_+}{\vr_+}-1\right)$. Indeed, by a simple calculation using \eqref{eq:beta2} we have
	\begin{equation}
	\frac{v_++\beta}{v_+-\beta}\frac{\vr_1-\vr_+}{\vr_+}-1 = \frac{\vr_1}{\vr_+}\frac{v_+ + \mu_1}{v_+ - \mu_1}.
	\end{equation}
	Plugging in \eqref{eq:nup2} we obtain
	\begin{align}
	\frac{\vr_1}{\vr_+}\frac{v_+ + \mu_1}{v_+ - \mu_1} &= \frac{\vr_1}{\vr_+}\left(-\frac{3v_++v_-}{\frac{\vr_1 u}{\vr_1-\vr_+} + \frac{p_--p_+}{\vr_+ u}} - 1\right) \nonumber \\
	&= \frac{\vr_1}{\vr_+}\left(-\frac{4v_++u}{\frac{\vr_1 u}{\vr_1-\vr_+} + \frac{p_--p_+}{\vr_+ u}} - 1\right) \label{eq:693} \\
	&= \frac{\vr_1}{\vr_+}\left(-\frac{4v_--3u}{\frac{\vr_1 u}{\vr_1-\vr_+} + \frac{p_--p_+}{\vr_+ u}} - 1\right) \label{eq:694}
	\end{align}
	If we assume $v_-$ is fixed we see from \eqref{eq:694} that the expression on the left hand side has a finite nonzero limit as $u \rightarrow \infty$, whereas if we assume $v_+$ is fixed, we conclude the same using \eqref{eq:693}.
	
	Altogether we have that the leading term in the expression \eqref{eq:ep2 11} as $u \rightarrow \infty$ is $p_1 \frac{\vr_1 - (1+2c_v)\vr_+}{\vr_1\vr_+}$. In particular if we assume $\vr_1 > (2c_v+1)\vr_+$, then $\ep_2 > 0$ provided $u$ is taken large enough.
\end{proof}

\subsubsection{Admissibility inequalities}
Nothing changes in the analysis of the admissibility inequalities \eqref{eq:adm_left2}--\eqref{eq:adm_right2} with respect to the case $R \neq 0$ and we can follow word by word section \ref{sss:AI}, in particular also here when we assume $R = 0$ we have 
\begin{equation}\label{eq:AD122}
p_1^{c_v} \geq \vr_1^{c_v+1}\max \bigg\{\frac{p_-^{c_v}}{\vr_-^{c_v+1}},\frac{p_+^{c_v}}{\vr_+^{c_v+1}}\bigg\}.
\end{equation}
and we know that $p_1$ grows as $u^2$, so we can choose $u$ large enough in order to \eqref{eq:AD122} to be satisfied.

\section*{Acknowledgements}
The research leading to these results was partially supported by the European Research Council under the European Union's Seventh Framework Programme (FP7/2007-2013)/ ERC Grant Agreement 320078. H. Al Baba, O. Kreml and V. M\'acha were supported by the GA\v CR (Czech Science Foundation) project GJ17-01694Y in the general framework of RVO: 67985840.

\end{document}